%% file: main.tex
\documentclass[11pt,leqno]{article}
\usepackage{graphicx, amsfonts, amsthm, amsxtra, amssymb, verbatim, makeidx}
\usepackage{subeqnarray, relsize}
\usepackage[mathscr]{euscript}
\usepackage[english]{babel}
\usepackage[fixlanguage]{babelbib}

\usepackage[utf8]{inputenc}
\usepackage[english]{babel}

\usepackage{wrapfig}
\usepackage{amssymb, amsmath, amsthm}
\usepackage{graphicx}
\usepackage{color}
\usepackage{amssymb}
\usepackage{url}
\usepackage{pdfpages}
\usepackage{fancyhdr}
\usepackage{hyperref}
\usepackage{subfig}
\usepackage{titlesec}
\usepackage{enumerate}
\usepackage{comment}
\usepackage{bigints}
\usepackage{diagbox}
\usepackage{cite}
\usepackage[fixlanguage]{babelbib}

\makeindex
\makeglossary

\theoremstyle{definition}
\newtheorem{definition}{\bf Definition}
\newtheorem{theorem}{Theorem}

\newtheorem{cor}{Corollary}
\newtheorem{prop}{Proposition}

\newtheorem{remark}{Remark}

\newtheorem{example}{Example}

\begin{document}

\def\R{\mathbb{R}}                   
\def\Z{\mathbb{Z}}                   
\def\Q{\mathbb{Q}}                   
\def\C{\mathbb{C}}                   
\def\N{\mathbb{N}}                   
\def\uhp{{\mathbb H}}                
\def\A{\mathbb{A}}

\def\P{\mathbb{P}}
\def\GHod{\text{GHod}}
\def\SHod{\text{SHod}}
\def\Hod{\text{Hod}}

\def\ker{{\rm ker}}              
\def\GL{{\rm GL}}                
\def\ker{{\rm ker}}              
\def\coker{{\rm coker}}          
\def\im{{\rm Im}}               
\def\coim{{\rm Coim}}            

\def\End{{\rm End}}              
\def\rank{{\rm rank}}                
\def\gcd{{\rm gcd}}                  

\begin{center}
{\LARGE\bf  Periods of Hodge cycles and special values of the Gauss' hypergeometric function
}
\\
\vspace{.25in} {\large {\sc Jorge Duque Franco}}\footnote{
Instituto de Matem\'atica Pura e Aplicada, IMPA, Estrada Dona Castorina, 110, 22460-320, Rio de Janeiro, RJ, Brazil,
{\tt georgy76@impa.br}}
\end{center}


\begin{abstract}
We compute periods of perturbations of a Fermat variety. This allows us to consider a subspace of the Hodge cycles defined by ``simple" arithmetic conditions. We explore some examples and give an upper bound for the dimension of this subspace. As an application, we find explicit expressions involving some Gauss' hypergeometric functions which are algebraic over the field of rational functions in one variable.
\end{abstract}

\section{Introduction}
\label{intro}
\input{intro.tex}

\section{Hodge cycles}
\label{sec2}
\input{s2.tex}

\section{Integration over joint cycles}
\label{sec3}
\input{s3.tex}

\section{Families of varieties and Hodge cycles}
\label{sec4.1}
\input{s4.tex}


\bibliographystyle{alpha}

\bibliography{ref}



\end{document}

%% file: intro.tex
The present work is devoted to the study of hypersurfaces with hypergeometric periods. We focus on a particular class, Fermat varieties perturbed by $P(y)=y(1-y)(\lambda-y)$. Periods (roughly speaking multiple integrals) are an essential part of Hodge theory that have their deepest origins in elliptic and abelian integrals. We do not aim to verify the Hodge conjecture in our examples, rather we would like to analyze transcendental properties of integration over Hodge cycles.

Deligne in $1982$ explored periods  of algebraic cycles. He proved that up to some constant power of $2\pi\sqrt{-1}$, the periods of algebraic cycles are algebraic with respect to the field of definition of the variety (see \cite{deligne1982hodge}). This would be also true for Hodge cycles if the Hodge conjecture holds true. In fact, Deligne proved that this property is satisfied by periods of Hodge cycles in classic Fermat varieties even though Hodge conjecture is unknown in this case. With this, he obtained algebraic relations between the values of the $\Gamma$-function on rational points. This same idea was elaborated in $2006$ by Reiter and Movasati with the family 

\begin{equation*}
    M_t: \;\;\ f(x):=x_1^3+x_2^3+\dots+x_5^3-x_1-x_2=t,
\end{equation*}
to obtain algebraic relations of values of the hypergeometric functions (see \cite{movasati2006hypergeometric}). For example, they proved that

\begin{equation*}
    e^{-\frac{5}{6}\pi i}\frac{F\left(\frac{5}{6},\frac{1}{6},1;\frac{27}{16}t^2 \right)}{F\left(\frac{5}{6},\frac{1}{6},1;1-\frac{27}{16}t^2\right)}
\end{equation*}
belongs to $\Q(\zeta_3)$ for some $t\in\overline{\Q}$ if and only if

\begin{equation*}
    \pi^2\frac{F\left(\frac{5}{6},\frac{1}{6},1;\frac{27}{16}t^2 \right)}{\Gamma\left(\frac{1}{3}\right)^3}, \;\; \pi^2\frac{F\left(\frac{5}{6},\frac{1}{6},1;1-\frac{27}{16}t^2 \right)}{\Gamma\left(\frac{1}{3}\right)^3} \in\overline{\Q}.
\end{equation*}
The above is satisfied if $t$ is any root of the following equations

\begin{equation*}
    91125t^4-54000t^2+256, \;\;\; 81000t^4-48000t^2-1,
\end{equation*}
see \cite{movasati2006hypergeometric} and the references therein.

In this paper, we elaborate these same ideas with the family 


\begin{equation*}
     M_{\lambda}: \;\;\ f(x):=x_1^{m_1}+\dots+x_n^{m_n}+y(y-1)(y-\lambda)=0.
\end{equation*}
We compute its periods and use them to give some algebraic hypergeometric functions (see equation (\ref{eq:alg1708})). This result is framed in Schwarz' work. Schwarz in \cite{schwarz1873ueber} was the first to classify hypergeometric functions which are algebraic over $\C(z)$. A crucial idea was to relate the hypergeometric equation with the monodromy group. In the famous \textbf{Schwarz' list}, Schwarz determines explicit criteria for the parameters of the irreducible hypergeometric equations such that the solutions are algebraic. In the same work, Schwarz also obtained a similar but not so famous criterion in the case of reducible hypergeometric equations (see also \cite[\S 5]{kimura1969riemann}).

A more general question raised by Wolfart in \cite[Problem 6]{sariouglu2007problem} is to determine the transcendence degree of the field extension of $\overline{\C(z)}$ generated by the hypergeometric functions $F(a,b,c; z)$ where $a$, $b$, $c\in\Q$ with some fixed denominator. Or even better to determine a complete list of algebraic dependence equations among these $F(a,b,c;z)$ over the field $\overline{\C(z)}$. Examples of such relations are Propositions  \ref{prop:algtra0906} and \ref{prop:algtra11407}, Schwarz' list and Gauss’ relations between contiguous hypergeometric functions (see \cite{vidunas2003contiguous} and the references therein). Up to the author's knowledge, Wolfart's problem remains open and without significant progress.  

Throughout the paper, $n$ will be an even number. Let $g(x):=x_1^{m_1}+\dots +x_{n}^{m_n}$, $m_i \geq 2$, and let $P$ be a degree $m=m_{n+1}$ polynomial with $m\geq 2$. Consider
$$f=g(x)+P(y),$$
and let $F$ be its quasi-homogenization inside the weighted projective space $\P^{(1,v)}$ where $v_j=\frac{\text{lcm}(m_1,\ldots,m_{n+1})}{m_j}$ for $j=1,\ldots,n+1$. Let $X$ be a desingularization of the weighted hypersurface $D:=\{F=0\}\subset \P^{(1,v)}$. We are interested in Hodge cycles of $X$ supported in the affine part $U:=\{f=0\}$. For this, we consider a parametric family. Let 

$$T:=\left\{t=(t_0,\dots,t_m)\in\C^{m+1}\bigg|t_m=1,\;\;\Delta(P_t)\neq0\text{ where }P_t:=\sum_{i=0}^mt_iy^i\right\}$$
be the space of polynomials of degree $m$ with nonzero discriminant, and let 

$$\mathcal{U}:=\{(x,y,t)\in\C^n\times\C\times T|\;f_t(x,y):=g(x)+P_t(y)=0\}$$ be the family of affine varieties parameterized by $T$. Thus, the projection $\pi:\mathcal{U}\longrightarrow T$ is a locally trivial $C^{\infty}$ fibration (see \cite[\S 7.4]{Mov} and the references therein). We denote by $U_t:=\pi^{-1}(t)=\{f_t=0\}\subset\C^{n+1}$ and $X_t$ be a desingularization of $D_t:=\{F_t=0\}\subset\P^{(1,v)}$ where $F_t$ is the quasi-homogenization of $f_t$.

We say that a cycle $\delta_{t_0}\in H_n({U}_{t_0},\Q)$ is a \textbf{generic Hodge cycle} if all perturbations $\delta_t$ of it in the family $T$ are Hodge cycles (see Definition \ref{def:GHC0107}). This space is denoted by $\GHod_n(X_{t_0},\Q)_0$. We consider a subspace of the  generic Hodge cycles space by imposing certain conditions, which we call the space of \textbf{strong generic Hodge cycles} and we denote it by $\SHod_n(X_{t_0},\Q)_0$ (see Definition \ref{def:SGH0403}). These cycles are supported in $\mathcal{U}$ and they do not depend on the desingularization, because desingularization is done outside of $\mathcal{U}$. Essentially $\mathcal{U}$ is unaffected by the blow-up process.

Our main result is an upper bound for the dimension of the space of strong generic Hodge cycles in certain cases. 

\begin{theorem}
\label{theo:main2708}
Let $X$ be a desingularization of the weighted hypersurface $D$ given by the quasi-homogenization $F$ of $f=g(x)+P(y)$, where $g(x)=x_1^{m_1}+\dots +x_{n}^{m_n}$, $m_i \geq 2$ and $P$ is a polynomial of degree $m\geq 2$. 

\begin{enumerate}[i.]
    \item For $m_1=\dots=m_{n-1}=2$ and $m\geq7$

\begin{equation*}
 \dim \SHod_n(X,\Q)_0\leq\left\lbrace
\begin{array}{ll}
m-1 & m_n \textup{ even,}\\
0  & m_n \textup{ odd.}\\
\end{array}
\right.
\end{equation*}

\item For $m_1=\dots=m_{n-2}=2$, $m_{n-1}$ prime, $\gcd(m_{n-1},m_n)=1$ and $\frac{1}{m_{n-1}}+\frac{1}{m}<\frac{1}{2}$, we have $\SHod_n(X,\Q)_0= 0$.

\item For $m_j,$ $j=1,\dots,n,$  different prime numbers, we have $\SHod_n(X,\Q)_0= 0$.
\end{enumerate}
\end{theorem}

In fact, the proof of the previous theorem provides a method to calculate a set of generators of $\SHod_n(X,\Q)_0$ even if $m<7$ in the first case and if $\frac{1}{m_{n-1}}+\frac{1}{m}\geq\frac{1}{2}$ for the second case. Using this, we get:

\begin{cor}
\label{cor:dimpaper2808}
Let $X$ be a desingularization of the weighted hypersurface $D$ given by the quasi-homogenization $F$ of $f=g(x)+P(y)$, where $g(x)=x_1^{m_1}+\dots +x_{n}^{m_n}$, $m_i \geq 2$ and $P$ is a polynomial of degree $m$. 

\begin{enumerate}[i.]
    \item For $m_1=\dots=m_{n-1}=2$ and $m=2,\dots,6$

$$\dim\SHod_n(X,\Q)_0\leq(m-1)\left(\sum_{\substack{2\leq d\leq\lfloor\frac{2m}{m-2}\rfloor\\ d|m_n}}\varphi\left(d\right) \right),$$
where $\varphi$ is the Euler's totient function. When $m=2$ means that $2\leq d\leq m_n$ and $d|m_n$. Therefore for $m=2$, $\dim\SHod_n(X,\Q)_0\leq(m_n-1)$.

\item For $m_1=\dots=m_{n-2}=2$, $m_{n-1}$ prime, $\gcd(m_{n-1},m_n)=1$ and $\frac{1}{m_{n-1}}+\frac{1}{m}\geq\frac{1}{2}$, we have 

$$\dim\SHod_n(X,\Q)_0\leq(m-1)(m_{n-1}-1)\left(\sum_{\substack{2\leq d\leq\lfloor\frac{mm_{n-1}}{mm_{n-1}-m-m_{n-1}}\rfloor\\ d|m_n}}\varphi\left(d\right) \right),$$
where $\varphi$ is the Euler's totient function.
\end{enumerate}
\end{cor}

We obtain algebraic hypergeometric functions by a different method than that used by Schwarz. 
For this, we restrict ourselves to the case  $P_{\lambda}(y)=y(1-y)(\lambda-y)$, and we compute the periods on explicit strong generic Hodge cycles. For example, we get 

\begin{equation}
\label{eq:alg1708}
    F\left(\frac{5}{6}, \frac{1}{6},\frac{5}{3};1-\lambda \right),\;\;\; F\left(\frac{7}{6}, \frac{-1}{6},\frac{7}{3};1-\lambda \right)\in\overline{\Q(\lambda)}.
\end{equation}        
The above is somewhat exceptional given that periods are usually transcendental. Other by-products of this work are examples of non-algebraic hypergeometric functions that satisfy algebraic relations between them. 
    
\begin{prop}
\label{prop:algtra0906}
The following expressions are in $\overline{\Q(\lambda)}:$

\begin{equation}
\label{eq:algHF1808}
    0\neq 6F\left(\frac{4}{3},-\frac{4}{3},\frac{8}{3};1-\lambda\right)(\lambda^2-\lambda+1)-\frac{2}{3}
F\left(\frac{4}{3},-\frac{1}{3},\frac{8}{3};1-\lambda\right)\left(\lambda+1 \right)(5\lambda^2-8\lambda+5),
\end{equation}

\begin{equation}
\label{eq:algHF18082}
        0\neq 2F\left(\frac{2}{3},-\frac{2}{3},\frac{4}{3};1-\lambda\right)-\frac{2}{3}
    F\left(\frac{2}{3},\frac{1}{3},\frac{4}{3};1-\lambda\right)\left(\lambda+1 \right),
\end{equation}

\begin{equation}
\label{eq:algHF18083}
\begin{split}
0\neq &4F\left(\frac{2}{3},-\frac{5}{3},\frac{4}{3};1-\lambda\right)(\lambda^2-\lambda+1)-
\frac{1}{3}F\left(\frac{2}{3},-\frac{2}{3},\frac{4}{3};1-\lambda\right)\left(\lambda+1 \right)(8\lambda^2-11\lambda+8)+\\
&F\left(\frac{2}{3},\frac{1}{3},\frac{4}{3};1-\lambda\right)\lambda(1-\lambda)^2,
\end{split}
\end{equation}

\begin{equation}
\label{eq:algHF18084}
\begin{split}
0\neq &6F\left(\frac{2}{3},-\frac{8}{3},\frac{4}{3};1-\lambda\right)(\lambda^2-\lambda+1)-
\frac{2}{3}F\left(\frac{2}{3},-\frac{5}{3},\frac{4}{3};1-\lambda\right)\left(\lambda+1 \right)(7\lambda^2-10\lambda+7)+\\
&2F\left(\frac{2}{3},-\frac{2}{3},\frac{4}{3};1-\lambda\right)\lambda(1-\lambda)^2,
\end{split}
\end{equation}
but each hypergeometric function in the expressions above is not algebraic over $\Q(\lambda)$. For a numerical verification of this proposition see \S \ref{sec:compu2408}.
\end{prop}

We can find the algebraic functions of the expressions in Proposition \ref{prop:algtra0906} using hypergeometric theory via Gauss' relations, see Remark \ref{rem:1610}. Proposition \ref{prop:algtra0906} suggests that the Hodge cycles in Theorem \ref{theo:main2708} and Corollary \ref{cor:dimpaper2808} should be absolute in the sense of Deligne, see \cite[\S 2]{deligne1982hodge}. Moreover, the algebraic functions in Proposition \ref{prop:algtra0906} might be used in order to construct the underlying algebraic cycles  explicitly, see \cite{movasati2020reconstructing}.

Similarly to Schwarz' work, Beukers and Heckman classified the generalised hypergeometric functions which are algebraic over $\C(z)$ in \cite{beukers1989monodromy}. On the other hand, meantime this article was being written, Movasati was able to obtain similar algebraicity properties of periods which are gathered in \cite[\S 16.9]{Mov}. Apparently these periods must be related in some way to the generalised hypergeometric functions described in \cite{beukers1989monodromy}, for instance via a pull-back. For the classification scheme of pull-back transformations between Gauss hypergeometric differential equations see \cite{vidunas2009algebraic}. 
\bigskip


\noindent\textbf{Acknowledgements.} 
I am deeply grateful to my advisor Hossein Movasati for his reading, suggestions and several useful conversations. I thank Stefan Reiter and Michael Dettweiler for hosting me at the University of Bayreuth and for providing such a stimulating environment to work. Furthermore, I would like to thank Roberto Villaflor for his helpful discussions, his comments and suggestions on the first version of this article. Funding was provided by CNPq (Grant No. 140607/2017-0).    

%% file: s2.tex
Throughout the text we will use $x_{n+1}$ and $y$ interchangeably. Let $f(x,y)=g(x)+P(y)$ be the polynomial given by $g(x)=x_1^{m_1}+\dots +x_{n}^{m_n}$ and a polynomial $P(y)$ of degree $m=m_{n+1}$ with non-zero discriminant. When we look at the usual compactification in the projective space of $U=\{f=0\}$ is usually not smooth. Thus, we will consider the compactification in the weighted projective space $\P^{(1,v)}$ with $v_j=\frac{\text{lcm}(m_1,\ldots,m_{n+1})}{m_j}$. In this case Steenbrink \cite[\S 4]{steenbrink1977intersection} describes how to construct an explicit basis for the cohomology of a given weighted hypersurface. This is just the generalization of the homogeneous smooth case given by Griffiths in \cite{griffiths1969periods} 
We use this explicit basis to state our definition of Hodge cycles.

Let $M$ be a smooth projective variety and $Y$ be a smooth hyperplane section of $M$. Writing the long exact sequence of the pair $(M,V)$, where $V=M\backslash Y$, and using the Thom-Leray isomorphism we have

$$\dots\rightarrow H_{n-1}(Y,\Z)\stackrel{\sigma}{\rightarrow} H_{n}(V,\Z)  \stackrel{i}{\rightarrow}H_{n}(M,\Z)\stackrel{\tau}{\rightarrow} H_{n-2}(Y,\Z)\rightarrow \cdots,$$
where the map $\tau$ is the intersection with $Y$. An element $\delta\in H_n(V,\Q)$ is called a \textbf{cycle at infinity} if $\delta\in Ker(H_n(V,\Q)\stackrel{i}{\rightarrow}H_n(M,\Q))$. We denote

\begin{equation*}
\begin{split}
    H_n(V,\Q)_{\infty}&:=Ker(H_n(V,\Q)\stackrel{i}{\rightarrow}H_n(M,\Q))\\
    &\cong Im(H_{n-1}(Y,\Z)\stackrel{\sigma}{\rightarrow}H_n(V,\Q)).
\end{split}    
\end{equation*}
We denote the primitive homology (dual to the primitive cohomology, see \cite[\S 5.7]{Mov}) by

\begin{equation*}
\begin{split}
    H_n(M,\Q)_{0}&:=Ker(H_n(M,\Q)\stackrel{\tau}{\rightarrow}H_{n-2}(Y,\Q)).
\end{split}    
\end{equation*}
Thus, we have

\begin{equation}
    \label{eq:primitive1503}
    \begin{split}
        H_n(M,\Q)_0&\cong\frac{H_n(V,\Q)}{Ker(H_n(V,\Q)\stackrel{i}{\rightarrow}H_n(M,\Q))}\\
        &=\frac{H_n(V,\Q)}{H_n(V,\Q)_\infty}.
    \end{split}
\end{equation}
On the other hand, the Hodge decomposition determines the Hodge filtration: $0=F^{n+1}\subset F^n\subset\dots\subset F^1\subset F^0=H_{dR}^n(M)$ with $F^k=F^kH_{dR}^n(M):=H^{n,0}+H^{n-1,1}+\dots+H^{k,n-k}$ where $H^{k,n-k}:=H^{k,n-k}(M)$, which allows us to define Hodge cycles. A cycle $\delta\in H_n(M,\Q)$ is called a \textbf{Hodge cycle} if

$$\int_{\delta}F^{\frac{n}{2}+1}=0.$$
We denote by $\Hod_n(M,\Q)$ the space of Hodge cycles in $H_n(M,\Q)$. Now, by \cite[Proposition 5.10]{Mov} and equation \eqref{eq:primitive1503} we have   

\begin{equation}
\label{eq:defprim1503}
\Hod_n(M,\Q)_0:=\Hod_n(M,\Q)\cap H_n(M,\Q)_0\cong\frac{\left\{\delta\in H_n(V,\Q)| \int_{\delta}F_0^{\frac{n}{2}+1}=0\right\}}{\left\{\delta\in H_n(V,\Q)| \int_{\delta}F_0^0=0 \right\}},    
\end{equation}
where $F_0^k=F^k\cap H_{dR}^{n}(M)_0$ is the corresponding Hodge filtration of the primitive cohomology. With this in mind, let us return to our case of interest. Let $F$ be the quasi-homogenization of $f$ given by

$$F(x_0,\dots,x_{n+1})=x_{0}^df\left(\frac{x_1}{x_{0}^{v_1}},\dots,\frac{x_{n+1}}{x_{0}^{v_{n+1}}} \right),$$
where $d:=\text{lcm}(m_1,\dots,m_{n+1}),$ $v_j=\frac{d}{m_j}$, $f(x,y)=g(x)+P(y)$ is the polynomial given by $g(x)=x_1^{m_1}+\dots +x_{n}^{m_n}$ and $P(y)$ is a polynomial of degree $m=m_{n+1}$ with non-zero discriminant. Thus $F$ is quasi-homogeneous in $\P^{(1,v)}$ with $v=(v_1,\dots,v_{n+1})$, and so it defines a weighted hypersurface $D$. We have

$$U:=\{f=0\}\subset D:=\{F=0\}\subset \P^{(1,v)}.$$ 

\begin{definition}
\label{def:Hodgecycle2708}
Let $X$ be a desingularization of the weighted hypersurface $D$ given by the quasi-homogenization $F$ of $f=g(x)+P(y)$. We define the space of primitive Hodge cycles as 

$$\Hod_n(X,\Q)_0:=\frac{\left\{\delta\in H_n(U,\Q)| \int_{\delta}res\left(\frac{\omega_\beta}{f^j}\right)=0, A_\beta<j, 1\leq j\leq\frac{n}{2} \right\}}{\left\{\delta\in H_n(U,\Q)| \int_{\delta}res\left(\frac{\omega_\beta}{f^j}\right)=0, A_\beta<j, 1\leq j\leq n+1 \right\}},$$
with $A_\beta=\sum_{j=1}^{n+1}(\beta_j+1)\frac{v_j}{d}$ and 
 $$\omega_\beta=x^\beta dx:=x_1^{\beta_1}\cdots x_{n+1}^{\beta_{n+1}}dx_1\wedge\dots\wedge dx_{n+1}.$$
\end{definition}

If $\P^{(1,v)}=\P^{n+1}$, it follows by \cite{steenbrink1977intersection} that this definition coincides with the classical definition of Hodge cycles (see equation \eqref{eq:defprim1503}). For instance, if $X$ is defined by $f=x_1^d+\dots+x_n^d+x_{n+1}^d+1$.

%% file: s3.tex
In this section we explain how to calculate periods on cycles in affine varieties. For a more detailed description,
the reader is referred to \cite{Mov,arnold1988singularities}.         

\subsection{Multiple Integrals for Fermat varieties}
\label{sec:fermat}

Let $m_1,m_2,\dots,m_{n}$ be integers bigger than one and consider the $(n-1)$-th affine Fermat variety

$$L_b:=\left\{x\in\C^{n}\Big| g(x)=b \right\} \subset\C^{n},$$
where  $g=x_1^{m_1}+\dots+x_{n}^{m_{n}},$ and $b\neq 0$. We denote $L:=L_1$. Let

$$\Delta^{n-1}:=\left\{(t_1,\cdots,t_{n})\in\R^{n}\Big|\;t_j\geq0,\; \sum_{j=1}^{n}t_j=1 \right\},$$
be the standard $(n-1)$-simplex and let $\zeta_{m_j}=e^{\frac{2\pi\sqrt{-1}}{m_j}}$ be an $m_j$-th primitive root of unity. For $\alpha\in J=I_{m_1}\times I_{m_2}\times \dots \times I_{m_{n}}$ with $I_m=\{0,\dots,m-2\}$ and $a\in\{0,1\}^{n}$, consider 

\begin{equation*}
 \begin{array}{cccc}
     \Delta_{\alpha+a}:&\Delta^{n-1}  & \longmapsto & L \\
     & (t_1,\dots,t_{n}) & \longmapsto & \left(t_1^{\frac{1}{m_1}}\zeta_{m_1}^{\alpha_1+a_1},\dots,t_{n}^{\frac{1}{m_{n}}}\zeta_{m_{n}}^{\alpha_{n}+a_{n}}\right).
\end{array}
\end{equation*}
The formal sum

$$\delta_{\alpha}:= \sum_{a}^{}(-1)^{\sum{_{i=1}^{n}}(1-a_i)} \Delta_{\alpha+a}$$
induces a non-zero element in $H_{n-1}(L,\Z).$ In fact

\begin{prop}
\label{prop:basis2608}
The cycles $\{\delta_{\alpha}^b\}_{\alpha\in J} $ are a basis for the $\Z$-module $H_{n-1}(L_b,\Z),$ with $\delta_{\alpha}^{b}=(\phi_b)_*(\delta_\alpha)$ and $\phi_b:L\rightarrow L_b$ is the biholomorphism given by 
$$\phi(x_1,\cdots,x_{n})=(b^{1/m_1}x_1,\cdots,b^{1/m_{n}}x_{n}),$$ 
where $b^{1/{m_j}}$ is a fixed $m_j$-th root of $b$.   
\end{prop}

\begin{proof}
See \cite[\S 7.9, Remark 7.1]{Mov} or \cite[$\S$2.9]{arnold1988singularities}
\end{proof}

The following proposition was first done by Deligne in \cite[Proposition 7.13 ]{deligne1982hodge} 
for the classical Fermat variety with $m_1=\dots=m_{n}.$ The general case is due to Movasati \cite[Proposition 15.1]{Mov}.

 \begin{prop}
\label{pro:fermat}
Let $g=x_1^{m_1}+\dots +x_{n}^{m_{n}}$, then

$$\int_{\Delta_{\alpha+a}}res\left(\frac{x^{\beta'}dx'}{g}\right)= \Lambda  B\left(\frac{\beta_1+1}{m_1},\dots,\frac{\beta _n+1}{m_n}\right),$$
where $x=(x',y)$, $dx=dx'\wedge dy=dx_1\wedge\dots\wedge dx_{n}\wedge dy$, $\beta=(\beta',\beta_{n+1})=(\beta_1,\dots,\beta_n,\beta_{n+1})$, $x^{\beta'}=x_1^{\beta_1}x_2^{\beta_2}\dots x_n^{\beta_n}=x'^{\beta'}$, $\Lambda$ is given by

$$\Lambda=(-1)^{n-1} \left( \prod_{j=1}^{n}\frac{1}{m_j}\right)  \left( \prod_{j=1}^{n}\zeta_{m_j}^{(\beta_j+1)(\alpha_j+a_j)}  \right),$$
and $B(b_1,\dots,b_n)$ is the multi parameter version of the beta function given by

\begin{equation*}
\label{nom:beta}
\begin{split}
B(b_1,b_2,\cdots,b_{n})&:=\frac{\Gamma(b_1)\cdots\Gamma(b_{n})}{\Gamma(b_1+\cdots+b_{n})},
\end{split}
\end{equation*}
with $\Gamma$ is the Gamma function. Therefore, for $\delta_\alpha\in H_{n-1}(L,\Z)$ we have

\begin{equation*}
\begin{split}
\int_{\delta_{\alpha}}res\left( \frac{x^{\beta'}dx'}{g}\right) &=  
\frac{(-1)^{n-1}}{\prod_{j=1}^{n} m_j} \prod_{j=1}^{n}\left(\zeta_{m_j}^{(\alpha_j+1)(\beta_j+1)} -\zeta_{m_j}^{\alpha_j(\beta_j+1)}\right)
B\left(\frac{\beta_1+1}{m_1},\dots,\frac{\beta_n+1}{m_n}\right).
\end{split}
\end{equation*}
\end{prop}

\begin{remark}
Via the biholomorphic map $\phi_b$ the periods of $L_b$ are given by

\begin{equation}
\label{fermat1}
\int_{\delta_{\alpha}^b}res\left( \frac{x^{\beta'}dx'}{g}\right)=b^{\sum_{j=1}^{n}\frac{\beta_j+1}{m_j}-1}\int_{\delta_{\alpha}} res\left( \frac{x^{\beta'}dx'}{g}\right).
\end{equation}
\end{remark}

\subsection{Joint cycles} 

Let $P(y)$ be a polynomial and let $C$ be the set of its critical values with $0,1\notin C$. This implies that the variety 

$$U:=\{(x,y)\in \C^{n}\times \C|P(y)=-g(x) \} $$
is smooth. Fix a regular value $b\in \mathbb{C}\setminus (C\cup 0)$ of $P$. Let $\delta_{1b}\in H_{0}(P^{-1}(b),\Z)$ and $\delta_{2b}\in H_{n-1}(-g^{-1}(b),\Z)$ be two vanishing cycles and $t_s$, $s\in[0,1]$ be a path in $\C$ such that it starts from a point in $C$, crosses b and ends in $0$ (the only critical value of $g$), and never crosses $C\cup 0$ except by the mentioned cases. We assume that $\delta_{1b}$ vanishes along $t^{-1}$ when $s$ tends to $0$ and $\delta_{2b}$ vanishes along $t$ when s tend to $1$. The following object 

$$\delta_1*\delta_2=\delta_1*_t\delta_2:=\bigcup_{s\in[0,1]}\delta_{1t_s}\times \delta_{2t_s}$$ 
induces a cycle in $H_{n}(U,\Z)$ and it is called the \textbf{joint cycle} of $\delta_{1b}$ and $\delta_{2b}$ along $t$. 

The set $\{\delta^{-1}_\alpha\}_{\alpha\in J}$ described in Proposition \ref{prop:basis2608} is a basis of vanishing cycles for $H_{n-1}(\{g=-1\},\Z)$. Take a basis $\{\delta_k\}_{k=0}^{m-2}$ of vanishing cycles for $H_0(\{P=1\},\Z)$. The joint cycles of these two basis satisfy

\begin{theorem}
\label{thm:evanescent}
The $\Z$-module $H_{n}(U,\Z)$ is freely generated by

$$\delta_k*\delta_{\alpha}^{-1},\;\; k=0,\cdots,m-2,\;\; \alpha\in J.$$
\end{theorem}

\begin{proof}
See \cite[Theorem 7.4]{Mov} or \cite[Theorem 2.9]{arnold1988singularities} for a proof in more general context. 
\end{proof}

For $t_0\in T$ and $t\in T$ in a neighborhood of $t_0$, the monodromy of  $\delta_k^{t_0}*\delta_{\alpha}^{-1}\in H_n(U_{t_0},\Z)$ is given by $\delta_k^{t}*\delta_{\alpha}^{-1}\in H_n(U_{t},\Z)$ where $\delta_k^t$ is the monodromy of $\delta_k^{t_0}$ in the family     

$$\mathcal{V}:=\{(y,t)\in\C\times T|\;P_t(y)=1\}.$$
Now, we describe how to reduce a higher dimensional integral to a lower-dimensional one.

\begin{prop}
\label{inte}
The integral over cycle $\delta_k*\delta_{\alpha}^{-1}$ is given by
\begin{equation*}
\int_{\delta_k*\delta_{\alpha}^{-1}}res\left( \frac{\omega_\beta}{P(y)+g(x)}\right)=\left\lbrace
\begin{array}{ll}
\frac{p(\{-g=1\},\beta',\delta_{\alpha}^{-1})}{p(\{z^q=1 \},\gamma,\delta)}\int_{\delta_k*\delta}
res\left(\frac{y^{\beta_{n+1}}z^{\gamma}dy\wedge dz}{P-z^q}\right) & 
A_{\beta'}\notin\N  \\
\\
\frac{p(\{-g=1\},\beta',\delta_{\alpha}^{-1})}{b^{A_\beta-1}}\int_{\widetilde{\delta}_k}
\widetilde{\omega} & A_{\beta'}\in\N 
\end{array},
\right.
\end{equation*}
where $\beta=(\beta',\beta_{n+1})$, $\delta=[\zeta_q]-[1]\in H_0(\{z^q=1\},\Z)$, $q$ and $\gamma$ are given by the equality $A_{\beta'}:=\sum_{i=1}^{n}\frac{\beta_i+1}{m_i}=\frac{\gamma+1}{q}$, and $p(\{z^q=1 \},\gamma,\delta)=\frac{\zeta_q^{\gamma+1}-1}{q}$, in the first case. In the second case, $\widetilde{\delta}_k\in H_0(\{P=0\},\Z)$ is the monodromy of $\delta_k$ along the path $t_s$ and $\widetilde{\omega}$ is the function such that $d\widetilde{\omega}=P(y)^{A_{\beta'}-1}y^{\beta_{n+1}} dy.$ In both cases

\begin{equation*}
    p(\{-g=1\},\beta',\delta_{\alpha}^{-1}):=\int_{\delta_{\alpha}^{-1}}res\left( \frac{x^{\beta'}dx'}{g}\right),
\end{equation*}
where $x^{\beta'}=x_1^{\beta_1}x_2^{\beta_2}\dots x_n^{\beta_n}$ and $dx'=dx_1\wedge\dots\wedge dx_{n}.$
\end{prop}

\begin{proof}
This is just a particular case of \cite[Proposition 13.9]{Mov}.
\end{proof}

As an illustration of the previous proposition we have

\begin{prop}
\label{prop:compint2006}
Let $f=g(x)+P(y)$ be a polynomial with $g(x)=x_1^{m_1}+x_{2}^{m_2}+\dots+x_{n}^{m_n}$, $m_j\geq 2$  and  $P(y)=y(1-y)(\lambda-y)$, $\lambda>1$. Consider the cycles $\delta_{0}:=[1]-[0], \delta_{1}:=[\lambda]-[1]\in H_0(\{P(y)=0\},\Z)$ and $t_s$ the straight line connecting one of the critical values of $P(y)$ with $0$.  If $A_{\beta'}\notin\N$, we have 

\begin{equation}
\label{eq:cal1710}
    \begin{split}
\int_{\delta_0*_t\delta_{\alpha}^{-1}} res\left(\frac{\omega_\beta}{f}\right)
=&\frac{\lambda^{A_{\beta'}-1}p(\{g=-1 \},\beta',\delta_{\alpha}^{-1})}{\zeta_{q}^{\gamma+1}-1}B\left(A_{\beta'}+\beta_{n+1},A_{\beta'}\right)\cdot \\ 
& \;\;\: F\left(A_{\beta'}+\beta_{n+1},1-A_{\beta'},2A_{\beta'}+\beta_{n+1};\frac{1}{\lambda}\right),
    \end{split}
\end{equation}

\begin{equation}
\label{eq:cal17101}
    \begin{split}
\int_{\delta_1*_t\delta_{\alpha}^{-1}} res\left(\frac{\omega_\beta}{f}\right)
=&\frac{(-1)^{A_{\beta'}-1}(\lambda-1)^{2A_{\beta'}-1}p(\{g=-1 \},\beta',\delta_{\alpha}^{-1})}{\zeta_{q}^{\gamma+1}-1}B\left(A_{\beta'},A_{\beta'}\right)\cdot \\ 
& \;\;\: F\left(A_{\beta'},1-A_{\beta'}-\beta_{n+1},2A_{\beta'};1-\lambda\right),
    \end{split}
\end{equation}
where $F(a,b,c;z)$ is the hypergeometric function and

\begin{equation*}
\begin{split}
p(\{g=-1 \},\beta',\delta_{\alpha}^{-1}) &=  
\frac{(-1)^{n+A_{\beta'}}}{\prod_{j=1}^{n} m_j} \prod_{j=1}^{n}\left(\zeta_{m_j}^{(\alpha_j+1)(\beta_j+1)} -\zeta_{m_j}^{\alpha_j(\beta_j+1)}\right)
B\left(\frac{\beta_1+1}{m_1},\dots,\frac{\beta_n+1}{m_n} \right).
\end{split}
\end{equation*}
\end{prop}

\begin{proof}
It is just an application of Propositions \ref{pro:fermat} and \ref{inte}. Note that the integral

$$\int_{\delta_k*\delta}
res\left(\frac{y^{\beta_{n+1}}z^{\gamma}dy\wedge dz}{P-z^q}\right)$$
is over a one-dimensional cycle in $\{P(y)=z^q\}$ and can be computed as a line integral by projecting $\delta_k*\delta$ onto the $y$-coordinate. For example if $k=0,$ we have 

\begin{equation*}
    \begin{split}
\int_{\delta_0*\delta}
res\left(\frac{y^{\beta_{n+1}}z^{\gamma}dy\wedge dz}{P-z^q}\right)&=\frac{1}{q}\int_{\delta_0*\delta}y^{\beta_{n+1}}z^{\gamma-q+1}dy\\
&=\frac{1}{q}\int_{C}y^{\beta_{n+1}}P(y)^{\frac{\gamma-q+1}{q}}dy\\
&=\frac{\lambda^{\frac{\gamma-q+1}{q}}}{q}\int_{0}^1s^{\beta_{n+1}}[s(1-s)(1-\frac{s}{\lambda})]^{\frac{\gamma-q+1}{q}}ds,\\
    \end{split}
\end{equation*} 
where $C$ is the path induced by the projection of $\delta_0*\delta$ that connects the roots $0$ and $1$ of the polynomial $P(y)$. A straightforward computation allows us to conclude the proof, taking into account that

\begin{equation*}
\label{hypint}
    F(a,b,c;z)=\frac{1}{B(a,c-a)}\int_0^1 t^{a-1}(1-t)^{c-a-1}(1-zt)^{-b}dt,
\end{equation*}
for $Re(c)>Re(a)>0$. 
\end{proof}

\subsection{Pole reduction}
We have seen how to calculate integrals of the residue of differential forms with pole of order one over joint cycles. Now we explain how to calculate the integrals of the residue of differential forms with arbitrary pole order. For this, we reduce the pole order of the differential form to order one and then we apply Proposition \ref{inte}. This is also known as Griffiths-Dwork method. We use an affine version of this method taken from \cite[\S 10,11]{Mov}.

Let $\Delta$ be the discriminant of the polynomial $P(y)$. We know there are polynomials $Q_1(y)$, $Q_2(y)$ such that 

\begin{equation}
\label{discriminant}
\Delta=Q_1\frac{\partial P}{\partial y}+PQ_2.    
\end{equation}

\begin{example}
\label{ex:discr1908}
Our main case of interest is when $P=y(1-y)(\lambda-y)$. In this case we have $\Delta=\lambda^2(1-\lambda)^2$. The polynomials $Q_1$, $Q_2$ satisfying equation \eqref{discriminant} are given by

$$Q_1(y)=a_\lambda y^2+b_\lambda y+c_\lambda,\;\; Q_2(y)=-3a_\lambda y+e_\lambda,$$
with

$$a_\lambda=2(\lambda^2-\lambda+1),\;\; b_\lambda=-(2\lambda^3-\lambda^2-\lambda+2),$$
$$c_\lambda=\lambda(1-\lambda)^2,\;\; e_\lambda=
4\lambda^3-3\lambda^2-3\lambda+4.$$
\end{example}

The following description of the differential form $\Delta dx$ will help us to reduce the pole order of a differential form with a pole along $\{f=0\}.$

\begin{prop}
There is an $n$-differential form $\xi$ such that

\begin{equation}
\label{discriminant1}
\Delta dx=df\wedge \xi+fQ_2dx,
\end{equation}
with $Q_2(y)$ as above.

\end{prop}

\begin{proof}
First remember that $x=(x',y)$ and $dx=dx'\wedge dy= dx_1\wedge\dots\wedge dx_n\wedge dy$. Consider

$$\xi=Q_1(y)dx'-Q_2(y)\eta' \wedge dy,$$
where $\eta':=\sum_{i=1}^{n}(-1)^{i-1}\frac{x_i}{m_i}\widehat{dx'_i},$  with  $\widehat{dx'_i}:=dx_1\wedge dx_2\wedge\dots\wedge dx_{i-1}\wedge dx_{i+1}\wedge\dots\wedge dx_{n}$ and $Q_1$ $Q_2$ satisfy \eqref{discriminant}.
\end{proof}

From equality \eqref{discriminant1}, it follows that there are $n$-differential forms $\xi_\beta$ such that

$$\Delta \omega_\beta=df\wedge \xi_\beta+f Q_2\omega_\beta,$$
namely $\xi_\beta=x^\beta\xi$. Thus

$$\Delta\frac{\omega_\beta}{f^j}=\frac{df\wedge \xi_\beta+f Q_2\omega_\beta}{f^j}=\frac{1}{j-1}\left(\frac{d\xi_\beta}{f^{j-1}}-d\left(\frac{\xi_\beta}{f^{j-1}}\right) \right)+\frac{Q_2\omega_\beta}{f^{j-1}}.$$
Using that $d(x'^{\beta'}\eta')=A_{\beta'}x'^{\beta'}dx'$, we have

\begin{equation*}
\begin{split}
 d\xi_\beta&=x'^{\beta'} \left(\beta_{n+1}y^{\beta_{n+1}-1}Q_1(y)+y^{\beta_{n+1}}\frac{\partial Q_1(y)}{\partial y}\right)dx'\wedge dy-A_{\beta'}Q_2(y)\omega_{\beta}\\
 &=\beta_{n+1}Q_1\omega_{\beta-(0',1)}+(Q'_1-A_{\beta'}Q_2)\omega_{\beta},
    \end{split}
\end{equation*}
where $\beta=(\beta',\beta_{n+1}),\;(0',1) \in I$ and $A_{\beta'}:=\sum_{i=1}^{n}\frac{\beta_i+1}{m_i}.$ Therefore, in $H_{dR}^{n+1}(\C^{n+1}\backslash U)$ we obtain the following formula that allows us to reduce the pole order:

\begin{equation}
\label{reduction}
\begin{split}
\left[\frac{\omega_\beta}{f^j} \right]&=\frac{1}{\Delta}\left[ \frac{1}{j-1}\left(\frac{d\xi_\beta}{f^{j-1}}\right)+\frac{Q_2 \omega_\beta}{f^{j-1}} \right]\\
&=\frac{1}{\Delta}\left[\frac{\beta_{n+1}Q_1\omega_{\beta-(0',1)}+(Q'_1-A_{\beta'}Q_2)\omega_{\beta}}{(j-1)f^{j-1}}+\frac{Q_2 \omega_\beta}{f^{j-1}} \right]\\
&=\frac{1}{\Delta}\left[\frac{\beta_{n+1}Q_1}{j-1}\frac{\omega_{\beta-(0',1)}}{f^{j-1}}+\left(\left(1-\frac{A_{\beta'}}{j-1} \right)Q_2+\frac{Q_1'}{j-1} \right)\frac{\omega_\beta}{f^{j-1}} \right].
\end{split}
\end{equation}
With this, we can compute integrals of differential forms with pole of a higher order. For example:

\begin{prop}
\label{prop:alg22206}
In the same context of Proposition \ref{prop:compint2006}, we have

\begin{equation*}
\begin{split}
\int_{\delta_0*_t\delta_{\alpha}^{-1}} res\left(\frac{\omega_\beta}{f^2}\right)
=&\frac{\lambda^{A_{\beta'}-3}p(\{g=-1 \},\beta',\delta_{\alpha}^{-1})}{(1-\lambda)^2(\zeta_{q}^{\gamma+1}-1)}B\left(
A_{\beta'}+\beta_{n+1}-1,A_{\beta'}\right)\times\\
&\left[\frac{(A_{\beta'}+\beta_{n+1}-1)_2}{(2A_{\beta'}+\beta_{n+1}-1)_2}F\left(A_{\beta'}+\beta_{n+1}+1,1-A_{\beta'},2A_{\beta'}+\beta_{n+1}+1;\frac{1}{\lambda} \right)\times\right.\\
& \left(3A_{\beta'}+\beta_{n+1}-1\right)a_\lambda+\\
& \frac{A_{\beta'}+\beta_{n+1}-1}{2A_{\beta'}+\beta_{n+1}-1} F\left(A_{\beta'}+\beta_{n+1},1-A_{\beta'},2A_{\beta'}+\beta_{n+1};\frac{1}{\lambda} \right)\times\\  
&\;\;\;\Big((1-A_{\beta'})e_\lambda+(1+\beta_{n+1})b_\lambda\Big)\\
& \left.+F\left(A_{\beta'}+\beta_{n+1}-1,1-A_{\beta'},2A_{\beta'}+\beta_{n+1}-1;\frac{1}{\lambda} \right) \beta_{n+1}c_\lambda \right].
\end{split}
\end{equation*}

\begin{equation*}
\begin{split}
\int_{\delta_1*_t\delta_{\alpha}^{-1}}
res\left(\frac{\omega_\beta}{f^2}\right)
&=\frac{(-1)^{A_{\beta'}-1}(\lambda-1)^{2A_{\beta'}-3}p(\{g=-1 \},\beta',\delta_{\alpha}^{-1})}{\lambda^2 (\zeta_{q}^{\gamma+1}-1)}B(A_{\beta'},A_{\beta'})\times\\
&\;\;\left[F\left(A_{\beta'},-(A_{\beta'}+\beta_{n+1}),2A_{\beta'};1-\lambda\right)\left(3A_{\beta'}+\beta_{n+1}-1\right)a_\lambda   \right.+\\
&\;\;\; F\left(A_{\beta'},-(A_{\beta'}+\beta_{n+1}-1),2A_{\beta'};1-\lambda\right)\left((1-A_{\beta'})e_\lambda+(1+\beta_{n+1})b_\lambda \right)\\
&\;\;+\left.F\left(A_{\beta'},-(A_{\beta'}+\beta_{n+1}-2),2A_{\beta'};1-\lambda\right)\beta_{n+1}c_\lambda \right],
\end{split}
\end{equation*}
with $a_\lambda,$ $b_\lambda,$ $c_\lambda,$ $e_\lambda$ as in Example \ref{ex:discr1908}, $(a)_n$ is the Pochhammer symbol defined by $$(a)_n:=a(a+1)\cdots(a+n-1),$$ and 

\begin{equation*}
\begin{split}
p(\{g=-1 \},\beta',\delta_{\alpha}^{-1}) &=  
\frac{(-1)^{n+A_{\beta'}}}{\prod_{j=1}^{n} m_j} \prod_{j=1}^{n}\left(\zeta_{m_j}^{(\alpha_j+1)(\beta_j+1)} -\zeta_{m_j}^{\alpha_j(\beta_j+1)}\right)
B\left(\frac{\beta_1+1}{m_1},\dots,\frac{\beta_n+1}{m_n} \right).
\end{split}
\end{equation*}
\end{prop}


\subsection{Pole increment}
\label{sec:poleinc1803}

We start this section by giving a criterion for when a differential form in the affine part actually comes from a differential form in the compactification. 

\begin{prop}
\label{prop:pole1102}
If $A_\beta=k\in\N$, the meromorphic form $\frac{\omega_\beta}{f^k}$ has pole of order one at infinity. If $A_\beta<k\in\N$, the meromorphic form $\frac{\omega_\beta}{f^k}$ has no pole at infinity.
\end{prop}

\begin{proof}
See \cite[Lemma 2]{steenbrink1977intersection} or \cite[Proposition 11.4]{Mov}.
\end{proof}

How to know if the meromorphic form $\frac{\omega_\beta}{f^k}$, $A_\beta>k$ comes from a form in the compactification? For this, we increment the pole order and apply the previous proposition. To increment the pole order we reproduce \cite[Proposition 11.3]{Mov}. Consider $\eta:=\sum_{i=1}^{n+1}(-1)^{i-1}\frac{x_i}{m_i}\widehat{dx_i}$ and $\eta_\beta=x^\beta\eta$. Observe that $d\eta_\beta=A_\beta\omega_\beta$, thus

\begin{equation*}
    \begin{split}
        \frac{\omega_\beta}{f^k}
        &=\frac{d\eta_\beta}{A_\beta f^k}\\
        &=\frac{1}{A_\beta}\left(\frac{kdf\wedge\eta_\beta}{f^{k+1}}+d\left(\frac{\eta_\beta}{f^k} \right) \right).
    \end{split}
\end{equation*}
Therefore in $H_{dR}^{n+1}(\C^{n+1}\backslash U)$ we have

\begin{equation*}
    \begin{split}
        \left[\frac{\omega_\beta}{f^k}\right]
        &=\frac{k}{A_\beta}\left[\frac{df\wedge\eta_\beta}{f^{k+1}}\right]\\
        &=\frac{k}{A_\beta}\left[\frac{f\omega_\beta+(h-f)\omega_\beta+d(f-h)\wedge\eta_\beta}{f^{k+1}}\right]\\
    \end{split}
\end{equation*}
and thus

\begin{equation}
\label{increase}
    \left[\frac{\omega_\beta}{f^k}\right]=\frac{A_\beta}{A_\beta-k}\left[\frac{(h-f)\omega_\beta+d(f-h)\wedge\eta_\beta}{f^{k+1}}\right],
\end{equation}
where $h$ is the last weighted homogeneous piece of $f$ and satisfies $dh\wedge\eta_\beta=h\omega_\beta$. In our case $h=g(x)+y^{deg(P(y))}=x_1^{m_1}+\dots +x_{n}^{m_n}+y^{deg(P(y))}.$ If $P(y)=\sum_{j=0}^{m_{n+1}}c_jy^j$ with $c_{m_{n+1}}\neq 0$ we have

   \begin{equation*}
\left[\frac{\omega_\beta}{f^k}\right]=\frac{A_\beta}{m_{n+1}(A_\beta-k)}\left[\frac{\sum_{j=0}^{m_{n+1}-1}(j-m_{n+1})c_j\omega_{\beta+(0',j)}}{f^{k+1}}\right].
\end{equation*}
Therefore, using the process of pole order increment, the meromorphic form $\frac{\omega_\beta}{f^k}$ with $A_\beta>k$ can be written as a finite sum  

\begin{equation}
\label{eq:increase2001}
    \left[\frac{\omega_\beta}{f^k}\right]=\sum C_{j} \left[\frac{\omega_{\beta^{l_j}}}{f^j}\right],
\end{equation}
with $A_{\beta^{l_j}}\leq j$, $k<j$, $\beta^{l_j}=(\beta_{1}^{l_j},\dots,\beta_{n+1}^{l_j})$ and $C_j\in\C$. Note that even when $A_{\beta^{l_j}}<j$ we can increment the pole order. We will stop the process of pole order increment the first time $A_{\beta^{l_j}}\leq j$ is satisfied.

\begin{remark}
\label{rem:incrpole2508}
For the polynomial $P(y)=y(y-1)(y-\lambda)$ the pole order increment looks like
$$\left[\frac{\omega_\beta}{f^k}\right]=\frac{A_\beta}{3(A_\beta-k)}\left[\frac{(1+\lambda)\omega_{\beta+(0',2)}-2\lambda\omega_{\beta+(0',1)}}{f^{k+1}} \right]. $$
Therefore, we can write the meromorphic form $\frac{\omega_\beta}{f^k}$ with $A_\beta>k$  as a finite sum  

\begin{equation}
\label{eq:increase2508}
    \left[\frac{\omega_\beta}{f^k}\right]=\sum C_{j} \left[\frac{\omega_{\beta+(0',k_j)}}{f^j}\right],
\end{equation}
with $k<j$ and it is the first time that  $A_{\beta+(0',{k_j})}\leq j$. This means that $A_{\beta+(0',{k_j})}\leq j$ and one step before reaching \eqref{eq:increase2508}, the differential form $\frac{\omega_{\beta+(0',k_j-1)}}{f^{j-1}}$ appears with $j-1<A_{\beta+(0',k_j-1)}$ or the differential form $\frac{\omega_{\beta+(0',k_j-2)}}{f^{j-1}}$ appears with $j-1<A_{\beta+(0',k_j-2)}$.    
\end{remark}

\begin{definition}
\label{def:goodform}
A meromorphic form $\frac{\omega_\beta}{f^k}$ is called  \textbf{good form} if $A_\beta<k$ or if $k<A_\beta\notin\N$ and the differential form  written as in equation (\ref{eq:increase2001}), satisfies $A_{\beta^{l_j}}<j$.
\end{definition}

Observe that a good form has no residue at infinity and hence it comes from an element in $H^n_{dR}(X)$.

\begin{example}
\label{exam:goodform3003}
Consider $f=g+P$ with $g(x)=x_1^{2}+x_{2}^{9}$ and $P(y)=y^3+ay^2+by+c$. The differential forms $\frac{\omega_\beta}{f}$ with $\beta=(0,\beta_2,\beta_3)$, $\beta_2=2,5$ and $\beta_3=0,1$ are good forms. Let us see it for $\beta_2=2$, we can write

$$\left[\frac{\omega_\beta}{f}\right]=\frac{A_\beta}{3(A_\beta-1)}\left[\frac{-a\omega_{\beta+(0',2)}-2b\omega_{\beta+(0',1)}-3c\omega_\beta}{f^2} \right], $$
where $0'=(0,0)$. If $\beta_3=0$ then $1<A_\beta, A_{\beta+(0',1)}, A_{\beta+(0',2)}<2$ and the above equation corresponds to equation (\ref{eq:increase2001}). If $\beta_3=1,$ then $1<A_\beta, A_{\beta+(0',1)}<2$ and $A_{\beta+(0',2)}>2$. Therefore we apply again the pole order increment to the differential form $\frac{\omega_{\beta+(0',2)}}{f^2}$ and we obtain

\begin{equation*}
    \begin{split}
        \left[\frac{\omega_\beta}{f}\right]=&\frac{-aA_\beta(A_\beta+2/3)}{9(A_\beta-1)(A_\beta-4/3))}\left[\frac{-a\omega_{\beta+(0',4)}-2b\omega_{\beta+(0',3)}-3c\omega_{\beta+(0',2)}}{f^3}\right]+\\
        &\frac{A_\beta}{3(A_\beta-1)}\left[\frac{-2\omega_{\beta+(0',1)}-3c\omega_\beta}{f^2} \right].
    \end{split}
\end{equation*}
with $2<A_{\beta+(0',j)}<3,$ $j=2,3,4.$
\end{example}

%% file: s4.tex
In this section, we introduce the notion of generic Hodge cycles which describe those Hodge cycles that remain Hodge when doing monodromy along any path in a family of varieties. We find generic Hodge cycles that are ``elementally" defined. As an application, we find algebraic expressions involving hypergeometric functions.

\subsection{Generic and strong generic Hodge cycles}
Let $X$ be a desingularization of the weighted hypersurface $D$ given by the quasi-homogenization $F$ of $f=g(x)+P(y)$, where $g(x)=x_1^{m_1}+\dots +x_{n}^{m_n}$, $m_i \geq 2$, $P(y):\C\rightarrow \C$ is a polynomial of degree $m=m_{n+1}\geq 2$. Let 

$$T:=\left\{t=(t_0,\dots,t_m)\in\C^{m+1}\bigg|t_m=1,\;\;\Delta(P_t)\neq0\text{ where }P_t:=\sum_{i=0}^mt_iy^i\right\}$$
be the space of polynomials of degree $m$ with nonzero discriminant, and let 

$$\mathcal{U}:=\{(x,y,t)\in\C^n\times\C\times T|\;f_t(x,y):=g(x)+P_t(y)=0\}$$
be the family of affine varieties parameterized by $T$. Thus the projection $\pi:\mathcal{U}\longrightarrow T$ is a locally trivial $C^{\infty}$ fibration (see \cite[\S 7.4]{Mov} and the references therein). We denote by $U_t:=\pi^{-1}(t)=\{f_t=0\}\subset\C^{n+1}$ and $X_t$ be a desingularization of $D_t:=\{F_t=0\}\subset\P^{(1,v)}$ where $F_t$ is the quasi-homogenization of $f_t$.

\begin{definition}
\label{def:GHC0107}
Fix $t_0\in T$, we say that $\delta_{t_0}\in H_n(U_{t_0},\Q)$ is a \textbf{generic Hodge cycle} if $\delta_t$ is a Hodge cycle of $X_{t}$, this means, $\delta_t\in \Hod_n(X_t,\Q)_0$  for all $t\in T$ and $\delta_{t}$ is the monodromy of $\delta_{t_0}$ along a path on $T$ that connects $t_0$ to $t$. We will denote this space by $\GHod_n(X_{t_0},\Q)_0.$ 
\end{definition}

On the other hand, by Definition \ref{def:Hodgecycle2708}, equation \eqref{reduction} and Proposition \ref{inte}, we have the following subspace of $\GHod_n(X_{t_0},\Q)_0.$

\begin{definition}
\label{def:SGH0403}
Consider the $\Q$-vector space

\begin{equation}
\label{eq:defA}
\begin{split}
\mathcal{A}:&=\left\{(n_{\alpha,k})\in\Q^{|I|} \left| \sum_{\alpha\in J}n_{\alpha,k}\int_{\delta_{\alpha}^{-1}}res\left( \frac{x^{\beta'}dx'}{g}\right)=0,\;\forall\beta \text{  s.t. } A_{\beta}<\frac{n}{2} \right. \right\}\\
&=\left\{(n_{\alpha,k})\in\Q^{|I|} \left| \sum_{\alpha\in J}
n_{\alpha,k}\prod_{j=1}^{n}\zeta
_{m_j}^{\alpha_j(\beta_j+1)}=0,\;\forall\beta' \text{  s.t. } A_{\beta'}<\frac{n}{2}-\frac{1}{m}  \right. \right\},
  \end{split}
  \end{equation}
 where $I=J\times\ I_m=I_{m_1}\times\dots\times I_{m_n}\times I_m$, $\beta=(\beta',\beta_{n+1})$, and $A_{\beta}=\sum_{j=1}^{n+1}\frac{\beta_j+1}{m_j}$. To obtain the equality in \eqref{eq:defA} we use Proposition \ref{pro:fermat} and that
 $$\left\{\beta'\Big|A_{\beta'}<\frac{n}{2}-\frac{1}{m}\right\}= \left\{\beta'\Big|A_{\beta}<\frac{n}{2},\;\beta=(\beta',\beta_{n+1})\right\}.$$
 The space of \textbf{strong generic Hodge cycles} is the image of $\mathcal{A}$ under the natural map
    
\begin{equation}
\label{eq:natmap2909}
     \begin{array}{ccc}
       \mathcal{A} &  \longrightarrow   & \Hod_n(X_{t_0},\Q)_0 \\
      (n_{\alpha,k}) & \longmapsto & \left[\sum_{k=0}^{m-2}\sum_{\alpha\in J}n_{\alpha,k} \delta_k*\delta_{\alpha}^{-1}\right].
     \end{array}
\end{equation}
We denote this space by $\SHod_n(X_{t_0},\Q)_0\subset \GHod_n(X_{t_0},\Q)_0.$ 
\end{definition}


\begin{remark}
\label{rem:sep2208}
It is easy to see that 
    \begin{equation*}
\mathcal{A}\cong\left\{(n_{\alpha})\in\Q^{|J|} \left|  
\begin{aligned}
\sum_{\alpha\in J}
n_{\alpha}\prod_{j=1}^{n}\zeta
_{m_j}^{\alpha_j(\beta_j+1)}=0,\;\forall\beta' \text{  s.t. } A_{\beta'}<\frac{n}{2}-\frac{1}{m}
\end{aligned}
 \right. \right\}^{m-1}, 
\end{equation*}
and therefore $\delta=\sum_{k=0}^{m-2}\delta^k$ is a strong generic Hodge cycle if and only if $\{\delta^k\}_{k=0,\dots,m-2}$ are strong generic Hodge cycles, with $\delta^k=\sum_{\alpha\in J}n_{\alpha,k} \delta_k*\delta_{\alpha}^{-1}$.
\end{remark}

We already have the necessary ingredients to prove Theorem \ref{theo:main2708}.

\begin{proof}[Proof of Theorem \ref{theo:main2708}]
For the first part consider 

\begin{equation*}
\mathcal{A}_{m_n,m}:=\left\{(n_{j})\in\Q^{m_n-1} \left|  
\begin{aligned}
\sum_{j=0}^{m_n-2}
n_{j}\zeta_{m_n}^{j(\beta_n+1)}=0,\;\forall\beta_n \text{  s.t. } \frac{\beta_n+1}{m_n}<\frac{1}{2}-\frac{1}{m}
\end{aligned}
 \right. \right\}. 
\end{equation*}
Note that in this case, with $g=x_1^2+\dots+x_{n-1}^2+x_n^{m_n}$, we have $\mathcal{A}=\mathcal{A}_{m_n,m}^{m-1}$. Thus, it is enough to prove that

\begin{equation*}
\mathcal{A}_{m_n,m} \cong \left\lbrace
\begin{array}{ll}
\Q & m_n \textup{ even,}\\
0  & m_n \textup{ odd.}\\
\end{array}
\right.
\end{equation*}
For this, consider 

\begin{equation}
\label{eq:S_mn}
 S_{m_n,m}:=\left\{1\leq e< m_n\left(\frac{1}{2}-\frac{1}{m}\right): \;e|m_n \right\}   
\end{equation}
and $Q_{(n_j)}(x)=\sum_{j=0}^{m_n-2}n_jx^j$. For each $(n_j)\in\mathcal{A}_{m_n,m}$ and $e\in S_{m_n,m}$ it is satisfied that $Q_{(n_j)}(\zeta^e)=0$ because $\frac{e}{m_n}<\frac{1}{2}-\frac{1}{m}$. This means for each $(n_j)\in\mathcal{A}_{m_n,m}$ we have 

\begin{equation*}
    \prod_{e\in S_{m_n,m}}\Phi_{m_n/e}(x)\bigg|Q_{(n_j)}(x),
\end{equation*}
where $\Phi_k$ is $k$th cyclotomic polynomial. The above implies that $\mathcal{A}_{m_n,m}\cong\Q^{m_n-1-N_{m_n,m}}$ with $N_{m_n,m}:=\sum_{e\in S_{m_n,m}}\varphi\left(\frac{m_n}{e}\right)$ and $\varphi$ is the Euler's totient function, via the isomorphism

\begin{equation*}
 \begin{array}{ccc}
     \mathcal{A}_{m_n,m} & \stackrel{\cong}{\longrightarrow}  & \Q[x]_{m_n-2-N_{m_n,m}} \\
     (n_j) & \longmapsto & \frac{Q_{(n_j)}(x)}{\prod_{e\in S_{m_n,m}}\Phi_{\frac{m_n}{e}}(x)}.
\end{array}
\end{equation*}
On the other hand, note that for $m\geq7$, we have that $S_{m_n,m}=\left\{1\leq e<\frac{m_n}{2}: \;e|m_n \right\}$, and using that $\sum_{e|m_n}\varphi(m_n/e)=m_n$ we get

\begin{equation*}
N_{m_n,m} = \left\lbrace
\begin{array}{ll}
m_n-2 & m_n \textup{ even,}\\
m_n-1  & m_n \textup{ odd.}\\
\end{array}
\right.
\end{equation*}
This allows us to conclude the proof in the first case.

The idea of the proof of the second case is similar in spirit to the first case.  Consider 

\begin{equation*}
\mathcal{B}_{m_{n-1},m_n,m}:=\left\{(n_{j})\in\Q^{m_n-1} \left|  
\begin{aligned}
\sum_{j=0}^{m_n-2}
n_{j}\zeta_{m_n}^{j(\beta_n+1)}=0,\;\forall\beta_n \text{  s.t. } \frac{\beta_n+1}{m_n}<1-\frac{1}{m}-\frac{1}{m_{n-1}}
\end{aligned}
 \right. \right\}, 
\end{equation*}

\begin{equation*}
\mathcal{A}_{m_{n-1},m_n,m}:=\left\{(n_{ij})\in\Q^{(m_{n-1}-1)(m_n-1)} \left| \begin{aligned}
&\sum_{i=0}^{m_{n-1}-2}\sum_{j=0}^{m_n-2}
n_{ij}\zeta_{m_{n-1}}^{i(\beta_{n-1}+1)}\zeta_{m_n}^{j(\beta_n+1)}=0,\\
&\forall (\beta_{n-1},\beta_n) \text{  s.t. }\frac{\beta_{n-1}+1}{m_{n-1}}+\frac{\beta_{n}+1}{m_n}<1-\frac{1}{m}
\end{aligned}
 \right. \right\}.
\end{equation*}
Note that in this case, with $g=x_1^2+\dots+x_{n-2}^2+x_{n-1}^{m_{n-1}}+x_n^{m_n}$, we have $\mathcal{A}=\mathcal{A}_{m_{n-1},m_n,m}^{m-1}$. As $\gcd(m_{n-1},m_n)=1$ and $m_{n-1}$ is a prime number, we have $[\Q(\zeta_{m_{n-1}},\zeta_{m_n}):\Q(\zeta_{m_n})]=m_{n-1}-1$ and therefore the $\Q(\zeta_{m_n})$-vector space $\Q(\zeta_{m_{n-1}},\zeta_{m_n})$ has basis $\{ 1,\zeta_{m_{n-1}},\dots,\zeta_{m_{n-1}}^{m_{n-1}-2}\}$. This implies that 

\begin{equation*}
\sum_{i=0}^{m_{n-1}-2}\sum_{j=0}^{m_n-2}
n_{ij}\zeta_{m_{n-1}}^{i(\beta_{n-1}+1)}\zeta_{m_n}^{j(\beta_n+1)}=0	\Longleftrightarrow \sum_{j=0}^{m_n-2}
n_{ij}\zeta_{m_n}^{j(\beta_n+1)}=0 \textrm{  for each  } 0\leq i\leq m_{n-1}-2.
\end{equation*}
Therefore 

\begin{equation*}
\begin{split}
\mathcal{A}_{m_{n-1},m_n,m}&\cong\left\{(n_{ij})\in\Q^{(m_{n-1}-1)(m_n-1)}\left|
\begin{aligned}
    &(n_{ij})_{j=0}^{m_n-2}\in\mathcal{B}_{m_{n-1},m_n,m},\\
    &\textrm{  for every  } 0\leq i\leq m_{n-1}-2 
\end{aligned}
\right.\right\}\\
&\cong(\mathcal{B}_{m_{n-1},m_{n},m})^{m_{n-1}-1}.
\end{split}
\end{equation*}
From the above, it is enough to prove that $\mathcal{B}_{m_{n-1},m_n,m}=0,$ when $\frac{1}{m_{n-1}}+\frac{1}{m}<\frac{1}{2}$. For this, we proceed exactly the same as in the first case. 

It remains for us to prove the third case. For this consider

\begin{equation*}
\mathcal{A}_{(m_j)}\cong\left\{(n_{\alpha})\in\Q^{|J|} \left|  
\begin{aligned}
\sum_{\alpha\in J}
n_{\alpha}\prod_{j=1}^{n}\zeta_{m_j}^{\alpha_j(\beta_j+1)}=0,\;\forall\beta' \text{  s.t. } A_{\beta'}<\frac{n}{2}-\frac{1}{m}
\end{aligned}
 \right. \right\}, 
\end{equation*}
where $J=\prod_{j=1}^n I_{m_j}$, $I_{m_j}=\{0,1,\dots,m_j-2 \}$ and $A_{\beta'}=\sum_{j=1}^{n}\frac{\beta_j+1}{m_j}$. Observe that $\mathcal{A}=\mathcal{A}_{(m_j)}^{m-1}$. It is sufficient to show that $\mathcal{A}_{(m_j)}=0$. As $m_1,\dots,m_n$ are different prime numbers, we have 

$$\Q(\zeta_{m_1},\dots,\zeta_{m_n})\cong\Q(\zeta_{\prod m_j}) \textrm{ and } [\Q(\zeta_{\prod m_j}):\Q]=\prod_{j=1}^n (m_j-1).$$ 
Moreover note that if $\alpha\neq\hat{\alpha}$, then $\prod_{j=1}^{n}\zeta_{m_j}^{\alpha_j}\neq\prod_{j=1}^{n}\zeta_{m_j}^{\hat{\alpha}_j}$. Therefore the $\Q$-vector space $\Q(\zeta_{m_1},\dots,\zeta_{m_n})$ has basis $\left\{\prod_{j=1}^{n}\zeta_{m_j}^{\alpha_j}\right\}_{\alpha\in J}$. This implies that 

\begin{equation*}
\begin{split}
\sum_{\alpha\in J}n_\alpha\prod_{j=1}^n\zeta_{m_j}^{\alpha_j}=0 &\Longleftrightarrow n_\alpha=0 \;\textrm{ for every } \alpha\in J, 
\end{split}
\end{equation*}
but the above is one of the conditions that satisfy the elements of $\mathcal{A}_{(m_j)}$, with $\beta'=(0,\dots,0)$.

\end{proof}

The proof of the third part of Theorem \ref{theo:main2708} allows us to deduce

\begin{cor}
Let $X$ be a desingularization of the weighted hypersurface $D$ given by the quasi-homogenization $F$ of $f=g(x)+P(y)$, where $g(x)=x_1^{m_1}+\dots +x_{n}^{m_n}$ and $P(y)=y^{m_{n+1}}+1$. If $m_j$, $j=1,\dots,n+1,$ are different prime numbers, then $\Hod_n(X,\Q)_0=0$.
\end{cor}

\begin{proof}
Let us consider $t_0=(1,0,\dots,0,1)\in T$. In this case $f_{t_0}=x_1^{m_1}+\dots +x_{n}^{m_n}+y^{m_{n+1}}+1$. In this particular case, we can rewrite equation \eqref{reduction} as
\begin{equation*}
    \left[\frac{\omega_\beta}{f^j}\right]=C_j\left[\frac{\omega_\beta}{f}\right],
\end{equation*}
with $C_j\in\Q.$ If $A_\beta\in\N$ and $A_\beta<j,$ then $C_j=0$. The last  equation up to a nonzero constant is equal to \cite[equation (11.13)]{Mov}. With this and using Proposition \ref{prop:basis2608} and \ref{pro:fermat}, we have in the definition of Hodge cycles (see Definition \ref{def:Hodgecycle2708})

$$\left\{\delta\in H_n(U_{t_0},\Q)\Big| \int_{\delta}res\left(\frac{\omega_\beta}{f_{t_0}^j}\right)=0, A_\beta<j, 1\leq j\leq\frac{n}{2} \right\}\cong \mathcal{A}_{(m_j)},$$
where
\begin{equation*}
\mathcal{A}_{(m_j)}\cong\left\{(n_{\alpha})\in\Q^{|I|} \left|  
\begin{aligned}
\sum_{\alpha\in I}
n_{\alpha}\prod_{j=1}^{n+1}\zeta_{m_j}^{\alpha_j(\beta_j+1)}=0,\;\forall\beta \text{  s.t. } A_{\beta}<\frac{n}{2}, A_{\beta}\notin\N
\end{aligned}
 \right. \right\}, 
\end{equation*}
with $I=\prod_{j=1}^{n+1} I_{m_j}$, $I_{m_j}=\{0,1,\dots,m_j-2 \}$ and $A_{\beta}=\sum_{j=1}^{n+1}\frac{\beta_j+1}{m_j}$. Similarly, as in the third part of the proof above, we have $\mathcal{A}_{(m_j)}=0,$ when $m_1,\dots,m_{n+1}$ are different prime numbers. In conclusion $\Hod_n(X_{t_0},\Q)_0=0$, where $X_{t_0}$ is  a desingularization of the weighted hypersurface $D_{t_0}$ given by the quasi-homogenization $F_{t_0}$ of $f_{t_0}$ and $m_j$, $j=1,\dots,n+1$ are different prime numbers. 
\end{proof}

The proof of Theorem \ref{theo:main2708} provides a method to calculate the rank of $\mathcal{A}_{m_n,m}$ and its generators when $m<7$ as well as provides a method to calculate the rank of $\mathcal{B}_{m_{n-1},m_n,m}$ and its generators when $\frac{1}{m_{n-1}}+\frac{1}{m}\geq\frac{1}{2}$. With this, we obtain Corollary \ref{cor:dimpaper2808}.

\begin{proof}[Proof of Corollary \ref{cor:dimpaper2808}]
With the notations of the previous proof,
we know that $\mathcal{A}=\mathcal{A}_{m_n,m}^{m-1}$ and $\mathcal{A}_{m_n,m}\cong\Q^{m_n-1-N_{m_n,m}}$ with $N_{m_n,m}:=\sum_{e\in S_{m_n,m}}\varphi\left(\frac{m_n}{e}\right),$ where $S_{m_n,n}$ is defined in (\ref{eq:S_mn}). Observe that 

$$N_{m_n,m}=m_n-\sum_{\substack{e\notin S_{m_n,m}\\ e|m_n}}
\varphi\left(\frac{m_n}{e}\right).$$
Therefore 

\begin{equation*}
    \begin{split}
\dim(\mathcal{A})&=(m-1)\left(\sum_{\substack{e\notin S_{m_n,m}\\e|m_n}}\varphi\left(\frac{m_n}{e}\right)-1 \right)=(m-1)\left(\sum_{\substack{m_n(\frac{1}{2}-\frac{1}{m})\leq e<m_n\\e|m_n}}\varphi\left(\frac{m_n}{e}\right) \right)\\        
&=(m-1)\left(\sum_{\substack{2\leq d\leq\frac{2m}{m-2}\\ d|m_n}}\varphi\left(d\right) \right),
    \end{split}
\end{equation*}
where $d=\frac{m_n}{e}.$ This proves the first part. For the second part, we know from the proof of Theorem \ref{theo:main2708} that $\mathcal{A}\cong\mathcal{B}_{m_{n-1},m_n,m}^{(m-1)(m_{n-1}-1)}$ and $\mathcal{B}_{m_{n-1},m_n,m}\cong\Q^{m_n-1-N_{m_{n-1},m_n,m}}$ with $N_{m_{n-1},m_n,m}:=\sum_{e\in S_{m_{n-1},m_n,m}}\varphi\left(\frac{m_n}{e}\right),$ where 

\begin{equation*}
\label{eq:S_mn1}
S_{m_{n-1},m_n,m}:=\left\{1\leq e< m_n\left(1-\frac{1}{m_{n-1}}-\frac{1}{m}\right); \;e|m_n\right\}. 
\end{equation*}
With this, we proceed as in the first part. The second part is similar to the first part.
\end{proof}

\subsection{Hodge numbers}
\label{sec:HN}
With the notations from the previous sections, let us consider $t_0=(1,0,\dots,0,1)\in T$. In this case $f_{t_0}=x_1^{m_1}+\dots +x_{n}^{m_n}+y^{m_{n+1}}+1$. The affine Fermat variety $\{f_{t_0}=0\}$ has a sequence of numbers related to the Hodge numbers of the compact smooth underlying variety $X_{t_0}$. Namely

\begin{equation*}
    h_0^{k-1,n-k+1}:=\#\{\beta\in I|\;k-1<A_\beta<k\}.
\end{equation*}
These number are symmetric, $h_0^{k-1,n-k+1}=h_0^{n-k+1,k-1}$, since the set $I$ is invariant under the transformation

$$\beta\longrightarrow\ m-\beta-2:=(m_1-\beta_1-2,\dots,m_{n+1}-\beta_{n+1}-2), $$
and therefore $A_{m-\beta-2}=n+1-A_\beta.$ These numbers satisfy $h^{p,q}=h_0^{p,q}$, for $p\neq q$, where $h^{p,q}=\dim H^{p,q}(X_{t_0}).$ In the remaining case $h^{\frac{n}{2},\frac{n}{2}}-h_0^{\frac{n}{2},\frac{n}{2}}$ depends on the desingularization of $D_{t_0}.$ For more details see \cite[\S 15.4]{Mov} and the references therein. The Hodge numbers do not change when the complex structure varies continuously. More precisely

\begin{theorem}
Let $\pi:\mathcal{X}\longrightarrow B$ be a family of complex manifolds (i.e. $\pi$ is proper and submersive) and assume that $\mathcal{X}_0$ is K\"ahler for some $0\in B$. Then for $b$ in a neighborhood of $0$ in B, the Hodge numbers of $\mathcal{X}_b$ are the same as the Hodge numbers of $\mathcal{X}_0$.  
\end{theorem}

\begin{proof}
See \cite[Proposition 9.20]{voisin2002hodge}. 
\end{proof}

The above theorem implies that the numbers $h_0^{k-1,n-k+1}$ are the same for every $t\in T$ in the fibration $\pi:\mathcal{U}\longrightarrow T,$ see also \cite[\S 7]{movasati2007mixed} and the references therein. 

Let $X$ be a desingularization of the weighted hypersurface $D$ given by the quasi homogenization $F$ of $f=g(x)+P(y)$, where  $g(x)=x_{1}^{m_{1}}+\dots+ x_{n_1}^{m_{n_1}}+x_{n_1+1}^2+x_{n_1+2}^2+\dots+x_{n_1+k_1}^2$ and $P(y)$ is a fixed polynomial of degree $m$. Here $n=n_1+k_1$. Observe that $\frac{k_1}{2}<A_\beta$ and this implies that $h_0^{k-1,n-k+1}=0=h_0^{n-k+1,k-1}$ for $k\leq\lfloor\frac{k_1}{2}\rfloor$. For example if $g=x_1^2+\dots+x_{n-1}^{2}+x_{n}^{m_n}$ or $g=x_1^2+\dots+x_{n-2}^2+x_{n-1}^{m_{n-1}}+x_{n}^{m_n}$, the sequence of numbers are surface-like:

$$0,\dots,0,h_0^{\frac{n}{2}+1,\frac{n}{2}-1},h_0^{\frac{n}{2},\frac{n}{2}},h_0^{\frac{n}{2}-1,\frac{n}{2}+1},0,\dots,0.$$
In other words, the Hodge structure of $H^n(X,\Z)$ has level $2$. Further, if $\mathbf{h}_0^{2,0},\mathbf{h}_0^{1,1}$ are the corresponding Hodge numbers for $g_2=x_1^2+x_2^{m_n}$ or $g_2=x_1^{m_{n-1}}+x_2^{m_n}$ respectively, we have $\mathbf{h}_0^{2,0}=h_0^{\frac{n}{2}+1,\frac{n}{2}-1}$ and $\mathbf{h}_0^{1,1}=h_0^{\frac{n}{2},\frac{n}{2}}.$

The variety induced by $f=g(x)+P(y)$, where $g=x_1^2+\dots+x_{n-1}^{2}+x_{n}^{m_n}$ and $P(y)=y(1-y)(\lambda-y)$, will be used constantly in the next section. In the following corollary, we calculate their Hodge numbers: $h_0^{\frac{n}{2}+1,\frac{n}{2}-1},$ $h_0^{\frac{n}{2}-1,\frac{n}{2}+1}$ and $h_0^{\frac{n}{2},\frac{n}{2}}.$ 

\begin{cor}
Let $X$ be a desingularization of the weighted hypersurface $D$ given by the  quasi-homogenization $F$ of $f=g(x)+P(y)$, where $g=x_1^2+\dots+x_{n-1}^{2}+x_{n}^{m_n}$ and $P(y)$ is a polynomial of degree $3$. We have

\begin{equation*}
h_0^{\frac{n}{2}+1,\frac{n}{2}-1}=h_0^{\frac{n}{2}-1,\frac{n}{2}+1} = \left\lfloor \frac{m_n-1}{6} \right\rfloor ,
\end{equation*}

\begin{equation*}
h_0^{\frac{n}{2},\frac{n}{2}} = m_n-2+ \left\lceil\frac{5m_n}{6} \right\rceil -  \left\lfloor\frac{m_n}{6} \right\rfloor .  
\end{equation*}
\end{cor} 

\begin{proof}
First, note that it is enough to prove it for the case $n=2$. Let us see one case, the other is analogous.

\begin{equation*}
\begin{split}
    h_0^{1,1}:=&\#\{\beta\in I|\;1<A_\beta<2\}\\
    =&\#\left\{\beta\in I\Big|\;\frac{1}{2}<\frac{\beta_2+1}{m_n}+\frac{\beta_3+1}{3}<\frac{3}{2}\right\}\\
    =&\#\left\{\beta_2\in I_{m_n}\Big|\;\frac{m_n}{2}<3(\beta_2+1)<3m_n\right\}+\#\left\{\beta_2\in I_{m_n}\Big|\;0<3(\beta_2+1)<\frac{5m_n}{2}\right\}\\
    =&m_n-1-  \left\lfloor\frac{m_n}{6} \right\rfloor+\left\lceil\frac{5m_n}{6} \right\rceil-1.
    \end{split}
\end{equation*}
\end{proof}

\subsection{Strong generic Hodge cycles and hypergeometric function}
In the rest of the document we will use $d$ instead of $m_n$. The proof of Theorem \ref{theo:main2708} provides us with a method to find strong generic Hodge cycles explicitly in the cases described. In the $2$-dimensional case, by Lefschetz $(1,1)$ theorem each Hodge cycle is algebraic. But the algebraic cycles satisfies

\begin{prop}[Deligne]
\label{prop:Deligne0109}
Let $X$ be a smooth projective variety. If $\delta\in H_m(X,\Q)$ is algebraic, then for every $\omega\in H^m_{dR}(X/k)$:

\begin{equation*}
    \frac{1}{(2\pi i)^{m/2}}\int_{\delta}\omega \in \overline{k}, 
\end{equation*}
where $X/k$ denotes the variety over a field $k\subset\C.$ 
\end{prop}

\begin{proof}
See \cite[Proposition 1.5]{deligne1982hodge}.
\end{proof}
We can find algebraic expressions involving hypergeometric functions using Proposition \ref{prop:Deligne0109} and the following fact

\begin{equation}
\label{eq:alg2301}
\frac{B\left(\frac{1}{2},\dots,\frac{1}{2},\frac{\beta_n+1}{d} \right)B\left(A_{\beta'}+k,A_{\beta'} \right)}{\pi^{\frac{n}{2}}}\in \overline{\Q},
\end{equation}
where $A_{\beta'}=\frac{n-1}{2}+\frac{\beta_n+1}{d}$ and  $k$, $\beta_n\in\N.$ Which is obtained using the properties of the beta $B$ and gamma $\Gamma$ functions.

In most of the following results, we will have the hypothesis that $A_{\beta'}\notin\N$. This assures that hypergeometric functions appear within the computations of the periods (see Propositions \ref{inte}, \ref{prop:compint2006}, \ref{prop:alg22206}). The following proposition allows us to partially re-obtain Schwarz' list in \cite{schwarz1873ueber}.    

\begin{prop}
\label{prop:alg0604}
Let $X$ be a desingularization of the weighted hypersurface $D$ given by the quasi-homogenization $F$ of $f=g(x)+P(y)$, where $g(x)=x_1^{2}+x_{2}^{d}$  and  $P(y)=y(1-y)(\lambda-y)$. Consider $\frac{\omega_\beta}{f}$ a good form with $A_{\beta'}=\frac{1}{2}+\frac{\beta_2+1}{d}\notin\N$ and $\beta=(\beta',\beta_3)=(\beta_1,\beta_2,\beta_3).$ Let 

\begin{equation*}
\begin{split}
    \delta^0=\sum_{j=0}^{d-2}n_{j,0}\delta_0*\delta_{j}^{-1},\;\;\;
    \delta^1=\sum_{j=0}^{d-2}n_{j,1}\delta_1*\delta_{j}^{-1}.
    \end{split}
\end{equation*}
If $\delta^0$ and $\delta^1$ are generic Hodge cycles then either

\begin{equation}
\label{eq:sum2908}
\sum_{j=0}^{d-2}n_{j,k}\zeta_{d}^{j(\beta_2+1)} ,\;\;\; k=0,1
\end{equation}
is zero or 

$$F\left(A_{\beta'}+\beta_{3},1-A_{\beta'},2A_{\beta'}+\beta_{3};\frac{1}{\lambda}\right) \text{ and } F\left(A_{\beta'},1-A_{\beta'}-\beta_{3},2A_{\beta'};1-\lambda\right)$$
are in $\overline{\Q(\lambda)}$.  

\begin{proof}
The fact that $\frac{\omega_\beta}{f}$ is a good form means that $res\left(\frac{\omega_\beta}{f}\right)\in H_{dR}^2(X/\Q)$ (see Proposition \ref{prop:pole1102}). Therefore by Lefschetz $(1.1)$ theorem and Proposition \ref{prop:Deligne0109} we have

\begin{equation*}
    \frac{1}{2\pi i}\int_{\delta}res\left(\frac{\omega_\beta}{f}\right) \in \overline{\Q(\lambda)}, 
\end{equation*}     
if $\delta$ is a generic Hodge cycle. The above integral is computed in Proposition \ref{prop:compint2006} and using equation \eqref{eq:alg2301} one gets the result.
\end{proof}
\end{prop}

In the context of Example \ref{exam:goodform3003}, using Theorem \ref{theo:main2708} we can describe the strong generic Hodge cycles. With this description, we observe that

$$\delta^1=n_0(\delta_1*\delta_{0}^{-1}+\delta_1*\delta_{3}^{-1}+\delta_1*\delta_{6}^{-1})+n_1(\delta_1*\delta_{1}^{-1}+\delta_1*\delta_{4}^{-1}+\delta_1*\delta_{7}^{-1})$$
is a strong generic Hodge cycle such that equation \eqref{eq:sum2908} is nonzero. Therefore applying the previous corollary to $\delta^1$ and the differential forms in Example \ref{exam:goodform3003}, we obtain that the hypergeometric functions in equation \eqref{eq:alg1708} are algebraic over $\Q(\lambda).$

The property in Proposition \ref{prop:Deligne0109} would be also true for Hodge cycles if the Hodge conjecture is true.
Deligne has proved this property for Hodge cycles in the usual Fermat variety, even though the Hodge conjecture is unknown. More explicitly
\begin{prop}[Deligne]
\label{prop:Deligne0204}
Let $X$ be a smooth projective variety defined by $x_1^d+\dots+x_{n+1}^d$. If $\delta\in H_m(X,\Q)$ is a Hodge cycle, then for every $\omega\in H^m_{dR}(X/k)$:

\begin{equation*}
    \frac{1}{(2\pi i)^{m/2}}\int_{\delta}\omega \in \overline{k}, 
\end{equation*}
where $X/k$ denotes the variety over a field $k\subset\C.$ \end{prop}

\begin{proof}
See \cite[Theorem 7.15]{deligne1982hodge} or \cite[Theorem 16.1]{Mov}. 
\end{proof}

In the same direction of the previous proposition, we have the following result

\begin{prop}
\label{prop:ndim2508}
Let $X_n$ be a desingularization of the weighted hypersurface $D_n$ given by the quasi-homogenization $F_n$ of $f_n=g_n(x)+P(y)$, where $g_n(x)=x_1^{2}+\dots+x_{n-1}^2+x_{n}^{d}$  and  $P(y)=y(1-y)(\lambda-y)$. Consider $\frac{\omega_{\beta}}{f_n}$ a good form with $A_{\beta'}\notin\N$. If $\delta\in H_n(X_n,\Q)$ is a strong generic Hodge cycle, we have 

\begin{equation*}
    \frac{1}{(2\pi i)^{n/2}}\int_{\delta}res\left(\frac{\omega_{\beta}}{f_n}\right) \in \overline{\Q(\lambda)}.
\end{equation*}
\end{prop}

\begin{proof}
The main idea of the proof is to use Lefschetz $(1,1)$ theorem in the $2$-dimensional case and to construct the $n$-dimensional integral from the $2$-dimensional integral. Consider $(n_{kj})\in\mathcal{A}$, where $\mathcal{A}$ is defined in (\ref{eq:defA}). This element induces the cycle      

\begin{equation*}
\delta^n=\delta^{n0}+\delta^{n1}=
\sum_{j=0}^{d-2}n_{j,0}\delta_0*\delta_{n,j}^{-1}+\sum_{j=0}^{d-2}n_{j,1}\delta_1*\delta_{n,j}^{-1},  \end{equation*}
with $\delta_{n,j}^{-1}\in H_{n-1}(\{g_n=-1\})$, which in turn induces a strong generic Hodge cycle. We know that $\delta^n$ is a strong generic Hodge cycle if and only if $\delta^{n0}, \delta^{n1}$ are strong generic Hodge cycles (see Remark \ref{rem:sep2208}). Now consider the differential form $\frac{\omega_{\hat{\beta}}}{f_2}$ with $\hat{\beta}=(0,\beta_n,\beta_{n+1})$. Observe that $\frac{\omega_{\hat{\beta}}}{f_2}=\frac{\omega_{\beta}}{f_n}$ for $n=2$ and that $\frac{n}{2}-1<A_\beta=A_{\hat{\beta}}+\frac{n}{2}-1$ for each $\beta$. An analysis similar to the proof of Proposition \ref{prop:goodform2909} allows us to deduce that if $\frac{\omega_{\beta}}{f_n}$ is a good form then $\frac{\omega_{\hat{\beta}}}{f_2}$ is a good form. Now, by Proposition \ref{prop:compint2006}, up to multiplication by a nonzero element of $\overline{\Q(\lambda)}$ we have

\begin{equation}
\label{eq:06042}
    \begin{split}
\frac{1}{(2\pi i)^{\frac{n}{2}}}\int_{\delta^{n0}} res\left(\frac{\omega_\beta}{f_n}\right)
&=\frac{1}{(2\pi i)^{\frac{n}{2}}}\left(\sum_{j=0}^{d-2}n_{j,0}\zeta_d^{j(\beta_n+1)}\right)B\left(\frac{1}{2},\dots,\frac{1}{2},\frac{\beta _n+1}{d}\right)\times\\
& \;\;\:B\left(A_{\beta'},A_{\beta'}\right) F\left(A_{\beta'}+\beta_{n+1},1-A_{\beta'},2A_{\beta'}+\beta_{n+1};\frac{1}{\lambda}\right).
    \end{split}
\end{equation}
On the other hand, up to multiplication by a nonzero element of $\overline{\Q(\lambda)}$

\begin{equation}
\label{eq:06043}
    \begin{split}
\frac{1}{2\pi i}\int_{\delta^{20}} res\left(\frac{\omega_{\hat{\beta}}}{f_2}\right)
&=\frac{1}{2\pi i}\left(\sum_{j=0}^{d-2}n_{j,0}\zeta_d^{j(\beta_n+1)}\right)B\left(\frac{1}{2},\frac{\beta _n+1}{d}\right)\times\\
& \;\;\:B\left(A_{\hat{\beta}'},A_{\hat{\beta}'}\right) F\left(A_{\hat{\beta}'}+\beta_{n+1},1-A_{\hat{\beta}'},2A_{\hat{\beta}'}+\beta_{n+1};\frac{1}{\lambda}\right).
    \end{split}
\end{equation}
Since $\delta^{20}$ induces a strong generic Hodge cycle of $X_2$, by Lefschetz $(1,1)$ theorem, $\delta^{20}$ is algebraic. By Proposition \ref{prop:Deligne0109} we have that $\frac{1}{2\pi i}\int_{\delta^{20}} res\left(\frac{\omega_{\hat{\beta}}}{f_2}\right)\in\overline{\Q(\lambda)}.$ If the integral in \eqref{eq:06043} is zero, then equation \eqref{eq:06042} is zero. If it is not zero, then equation \eqref{eq:alg2301} allows us to conclude that the hypergeometric function of equation \eqref{eq:06043} is algebraic over $\Q(\lambda)$. As $A_{\beta'}=A_{\hat{\beta}'}+\frac{n}{2}-1$ and $\frac{\omega_{\hat{\beta}}}{f_2}$ is a good form we have that the functions
 
$$F\left(A_{\beta'}+\beta_{n+1},1-A_{\beta'},2A_{\beta'}+\beta_{n+1};\frac{1}{\lambda}\right), \;\;\;F\left(A_{\hat{\beta}'}+\beta_{n+1},1-A_{\hat{\beta}'},2A_{\hat{\beta}'}+\beta_{n+1};\frac{1}{\lambda}\right) $$
are contiguous and irreducible. The hypergeometric function $F(a,b,c)$ is called \textbf{irreducible} if neither of $a$, $b$, $c-a$, $c-b$ is an integer. Two hypergeometric functions $F(a_1,b_1,c_1;z)$ and $F(a_2,b_2,c_2;z)$ are \textbf{contiguous} if the difference of their parameters

$$a_1-a_2, b_1-b_2, c_1-c_2$$
are integer numbers. Now by classical theory of the hypergeometric function, the hypergeometric function in \eqref{eq:06042} is equal to a $\Q(\lambda)$-linear combination of the hypergeometric functions in \eqref{eq:06043} and their derivatives (see \cite[Theorem 1.1]{beukers2007gauss}). In conclusion, the hypergeometric function in \eqref{eq:06042} is also algebraic over $\Q(\lambda)$. Now using equation \eqref{eq:alg2301} we conclude that \eqref{eq:06042} is algebraic over $\Q(\lambda).$ The same reasoning is valid for the integral $\frac{1}{(2\pi i)^{\frac{n}{2}}}\int_{\delta^{n1}} res\left(\frac{\omega_\beta}{f_n}\right)$, which allows us to conclude the result.


\end{proof}

The above argument does not work for a differential form with pole of order greater than one. However, we can still get a similar result using the same ideas with an extra hypothesis.

\begin{prop}
\label{prop:algform1807}
Let $X$ be a desingularization of the weighted hypersurface $D$ given by the quasi-homogenization $F$ of $f=g(x)+P(y)$, where $g(x)=x_1^{2}+\dots+x_{n-1}^2+x_{n}^{d}$  and  $P(y)=y(1-y)(\lambda-y)$. Consider $\frac{\omega_{\beta}}{f^k}$ a good form such that $\frac{\omega_{\beta}}{f}$ is also a good form,  and $A_{\beta'}\notin\N$. If $\delta\in H_n(X,\Q)$ is a strong generic Hodge cycle, we have 

\begin{equation*}
    \frac{1}{(2\pi i)^{n/2}}\int_{\delta}res\left(\frac{\omega_{\beta}}{f^k}\right) \in \overline{\Q(\lambda)}.
\end{equation*}
\end{prop}

\begin{proof}
Consider $(n_{kj})\in\mathcal{A}$, this element induces the cycle      
\begin{equation*}
\delta=\delta^{0}+\delta^{1}=
\sum_{j=0}^{d-2}n_{j,0}\delta_0*\delta_{j}^{-1}+\sum_{j=0}^{d-2}n_{j,1}\delta_1*\delta_{j}^{-1},   
\end{equation*}
with $\delta_{j}^{-1}\in H_{n-1}(\{g=-1\}),$ which in turn induces a strong generic Hodge cycle. We know that $\delta$ is a strong generic Hodge cycle if and only if $\delta^{0}, \delta^{1}$ are strong generic Hodge cycles (see Remark \ref{rem:sep2208}). By Proposition \ref{prop:compint2006}, up to multiplication by a nonzero element of $\overline{\Q(\lambda)}$ we have

\begin{equation}
\label{eq:algn02090}
    \begin{split}
\frac{1}{(2\pi i)^{\frac{n}{2}}}\int_{\delta^{0}} res\left(\frac{\omega_\beta}{f}\right)
&=\left(\sum_{j=0}^{d-2}n_{j,0}\zeta_d^{j(\beta_n+1)}\right)\frac{B\left(\frac{1}{2},\dots,\frac{1}{2},\frac{\beta _n+1}{d}\right)B\left(A_{\beta'},A_{\beta'}\right)}{(2\pi i)^{\frac{n}{2}}} F\left(a,b,c;\frac{1}{\lambda}\right),
    \end{split}
\end{equation}
where $a=A_{\beta'}+\beta_{n+1},$       $b=1-A_{\beta'},$         $c=2A_{\beta'}+\beta_{n+1}$. The fact that $\frac{\omega_{\beta}}{f}$ is a good form implies that $F\left(a,b,c;\frac{1}{\lambda}\right)$ is irreducible. An inductive argument allows us to prove, up to multiplication by a nonzero element of $\overline{\Q(\lambda)}$ that

\begin{equation}
\label{eq:algn0209}
\begin{split}
\frac{1}{(2\pi i)^{\frac{n}{2}}}\int_{\delta^{0}} res\left(\frac{\omega_\beta}{f^k}\right)
=&\left(\sum_{j=0}^{d-2}n_{j,0}\zeta_d^{j(\beta_n+1)}\right)\frac{B\left(\frac{1}{2},\dots,\frac{1}{2},\frac{\beta_n+1}{d} \right)B\left(
A_{\beta'},A_{\beta'}\right)}{(2\pi i)^{n/2}}\times \\
&\sum_{j}C_j(\lambda)F\left(a_j,b_j,c_j;\frac{1}{\lambda} \right),
\end{split}
\end{equation}
with $F\left(a_j,b_j,c_j;\frac{1}{\lambda} \right)$ contiguous to $F\left(a,b,c;\frac{1}{\lambda} \right)$ and $C_j(\lambda)\in\Q(\lambda)$. The first inductive step is for $k=2$. In this case we use Proposition \ref{prop:alg22206}. For the general case we apply pole order reduction (see \eqref{reduction}) and then the inductive hypothesis. Note that if the integral in \eqref{eq:algn02090} is zero, then integral in \eqref{eq:algn0209} is zero. 
Now, suppose that $\delta^{0}$ is a strong generic Hodge cycle and that the integral in \eqref{eq:algn02090} is nonzero. By Proposition \ref{prop:ndim2508} and equation \eqref{eq:alg2301} we have that $F\left(a,b,c;\frac{1}{\lambda} \right)\in\overline{\Q(\lambda)}$. Therefore $F\left(a_j,b_j,c_j;\frac{1}{\lambda} \right)$ are algebraic over $\Q(\lambda)$ (see \cite[Theorem 1.1]{beukers2007gauss}). With this and using equation \eqref{eq:alg2301} we conclude that \eqref{eq:algn0209} is algebraic over $\Q(\lambda).$ The same is valid for the cycle $\delta^{1}.$     
\end{proof}

\begin{remark}
Under the hypotheses of the previous proposition, the proof tells us that the hypergeometric functions that appear in the integral $\frac{1}{(2\pi i)^{n/2}}\int_{\delta}res\left(\frac{\omega_{\beta}}{f^k}\right)$ are algebraic.
\end{remark}

In the $2$-dimensional case, the previous result is independent of the hypothesis that    $\frac{\omega_\beta}{f}$ is a good form. What will be the nature of the hypergeometric functions that appear in the integral $\frac{1}{2\pi i}\int_{\delta^{j}} res\left(\frac{\omega_\beta}{f^k}\right),$ $j=0,1$, when $\frac{\omega_\beta}{f^k}$ is a good form and $\frac{\omega_\beta}{f}$ is not a good form? Exploring these integrals with $k=2$ we obtain Proposition \ref{prop:algtra0906}.

\begin{proof}[Proof of Proposition \ref{prop:algtra0906}]
Let $X$ be a desingularization of the weighted hypersurface $D$ given by the quasi-homogenization $F$ of $f=g(x)+P(y)$, where $g(x)=x_1^2+x_2^6$ and $P(y)=y(1-y)(\lambda-y)$. Consider the good form $\frac{\omega_\beta}{f^2}$ with $\beta=\left(0,4,0\right)$. Observe that $\frac{\omega_\beta}{f}$ is not a good form. In this case $\mathcal{A}=\Q^{2\times5}$, so every cycle in $H_2(U,\Q)$ induces an element in $\SHod_2(X,\Q)_0.$ Now consider the strong generic Hodge cycle induced by $\delta^1=\sum_{j=0}^{4}n_j\delta_1*\delta_{j}^{-1}$ with $n_0-n_3+n_2\neq0$ or $n_1-n_2+n_4\neq0$. This guarantees that

\begin{equation*}
\sum_{j=0}^{4}n_{j}\zeta_{6}^{5j}\neq 0,
\end{equation*}
since variety $X$ is $2$-dimensional, we have 

\begin{equation*}
    \frac{1}{2\pi i}\int_{\delta^1}res\left(\frac{\omega_{\beta}}{f^2}\right) \in \overline{\Q(\lambda)}.
\end{equation*}
Therefore by Proposition \ref{prop:alg22206} and equation \eqref{eq:alg2301} we conclude that 

\begin{equation*}
\begin{split}
\left [6F\left(\frac{4}{3},-\frac{4}{3},\frac{8}{3};1-\lambda\right)(\lambda^2-\lambda+1)-\frac{2}{3}
F\left(\frac{4}{3},-\frac{1}{3},\frac{8}{3};1-\lambda\right)\left(\lambda+1 \right)(5\lambda^2-8\lambda+5)\right]\in\overline{\Q(\lambda)}.
\end{split}
\end{equation*}
It remains to prove that $F\left(\frac{4}{3},-\frac{4}{3},\frac{8}{3};1-\lambda\right)$, $F\left(\frac{4}{3},-\frac{1}{3},\frac{8}{3};1-\lambda\right)\notin\overline{\Q(\lambda})$. For this, note that the above hypergeometric functions are reducible and their angular parameters $\lambda=1-c$, $\mu=c-a-b$, $\nu=a-b$ don't satisfy that exactly two of the numbers $\lambda+\mu+\nu,\; -\lambda+\mu+\nu,\; \lambda-\mu+\nu,\; \lambda+\mu-\nu$ are odd integers. This implies that $F\left(\frac{4}{3},-\frac{4}{3},\frac{8}{3};1-\lambda\right)$, $F\left(\frac{4}{3},-\frac{1}{3},\frac{8}{3};1-\lambda\right)\notin\overline{\Q(\lambda}),$ see \cite[Theorem 12.17, item (c)]{zolkadek2006monodromy}. 
To obtain the other expressions the reasoning is the same but using the differential forms $\frac{\omega_\beta}{f^2}$ with $\beta=\left(0,0,\beta_3\right)$, and $\beta_3=0,1,2$. 
\end{proof}

\begin{remark}
Let $X$ be a desingularization of the weighted hypersurface $D$ given by the quasi-homogenization $F$ of $f=g(x)+P(y)$, where $g(x)=x_1^2+x_2^d$, $6|d$ and  $P(y)=y(1-y)(\lambda-y)$. If for each $\beta_2\in\left\{\frac{5d}{6}-1,\frac{d}{6}-1\right\}$,  there is a strong generic Hodge cycle $\delta^1=\sum_{j=0}^{d-2}n_j\delta_1*\delta_{j}^{-1}$ on $X$ such that

\begin{equation*}
\sum_{j=0}^{d-2}n_{j}\zeta_{d}^{j(\beta_2+1)}\neq 0,
\end{equation*}
then we obtain exactly the same result of Proposition \ref{prop:algtra0906} using the differential forms $\frac{\omega_\beta}{f^2}$ with 
$\beta=\left(0,\frac{5d}{6}-1,0\right)$ and
$\beta=\left(0,\frac{d}{6}-1,\beta_3\right)$, $\beta_3=0,1,2$. 
\end{remark}

Using the same idea with the same $\beta$'s in the same variety of the proof of Proposition \ref{prop:algtra0906} and with the cycle $\delta^0=\sum_{j=0}^{4}n_j\delta_0*\delta_{j}^{-1}$ such that $n_0-n_3+n_2\neq0$ or $n_1-n_2+n_4\neq0$, we have

\begin{prop}
\label{prop:algtra11407}
The following expressions are in $\overline{\Q(\lambda)}:$

\begin{equation}
\label{eq:06051}
    0\neq \frac{3}{5}F\left(\frac{7}{3},-\frac{1}{3},\frac{11}{3};\frac{1}{\lambda}\right)(\lambda^2-\lambda+1)-\frac{2}{15}
F\left(\frac{4}{3},-\frac{1}{3},\frac{8}{3};\frac{1}{\lambda}\right)\left(\lambda+1 \right)(5\lambda^2-8\lambda+5),
\end{equation}

\begin{equation}
\label{eq:06052}
        0\neq F\left(\frac{5}{3},\frac{1}{3},\frac{7}{3};\frac{1}{\lambda}\right)-\frac{2}{3}
    F\left(\frac{2}{3},\frac{1}{3},\frac{4}{3};\frac{1}{\lambda}\right)\left(\lambda+1 \right),
\end{equation}

\begin{equation}
\label{eq:06053}
\begin{split}
0\neq &\frac{10}{7}F\left(\frac{8}{3},\frac{1}{3},\frac{10}{3};\frac{1}{\lambda}\right)(\lambda^2-\lambda+1)-
\frac{1}{6}F\left(\frac{5}{3},\frac{1}{3},\frac{7}{3};\frac{1}{\lambda}\right)\left(\lambda+1 \right)(8\lambda^2-11\lambda+8)+\\
&F\left(\frac{2}{3},\frac{1}{3},\frac{4}{3};\frac{1}{\lambda}\right)\lambda(1-\lambda)^2,
\end{split}
\end{equation}

\begin{equation}
\label{eq:06054}
\begin{split}
0\neq &\frac{24}{7}F\left(\frac{11}{3},\frac{1}{3},\frac{13}{3};\frac{1}{\lambda}\right)(\lambda^2-\lambda+1)-
\frac{10}{21}F\left(\frac{8}{3},\frac{1}{3},\frac{10}{3};\frac{1}{\lambda}\right)\left(\lambda+1 \right)(7\lambda^2-10\lambda+7)+\\
&2F\left(\frac{5}{3},\frac{1}{3},\frac{7}{3};\frac{1}{\lambda}\right)\lambda(1-\lambda)^2,
\end{split}
\end{equation}
but each hypergeometric function in the expressions above is not algebraic over $\Q(\lambda)$.
\end{prop}

\begin{remark}
\label{rem:1610}
The algebraic functions of the expressions in Propositions \ref{prop:algtra0906} and \ref{prop:algtra11407} can be found using Gauss' relations. Using the relation

\begin{equation}
\label{eq:1610}
    (c-b)F(a,b-1,c;z)+(2b-c-bz+az)F(a,b,c;z)+b(z-1)F(a,b+1,c;z)=0,
\end{equation}
with $a=\frac{2}{3}$, $b=\frac{1}{3}$ and $c=\frac{4}{3}$ we obtain that equation \eqref{eq:algHF18082} is equal to $\frac{2}{3}\lambda^{\frac{1}{3}}$. Using the latter together with equation \eqref{eq:1610} where $a=\frac{2}{3}$, $b=\frac{-2}{3}$ and $c=\frac{4}{3}$ we find that equation \eqref{eq:algHF18083} is equal to $\frac{1}{3}\lambda^{\frac{4}{3}}(\lambda+1).$ Similarly, we can see that equation \eqref{eq:algHF1808} is equal to $\frac{10}{3}\lambda^{\frac{5}{3}}$ and equation \eqref{eq:algHF18084} is equal to $\frac{2}{3}\lambda^{\frac{7}{3}}$. Now, with these same ideas and using the relation  

\begin{equation*}
\begin{split}
    \frac{a(b-c)}{c}zF\left(a+1,b,c+1;z \right) +\left((a-b)z+(c-1)\right) F\left(a,b,c;z\right)&\\
    -(c-1)F\left(a-1,b,c-1;z \right)&=0,
    \end{split}
\end{equation*}
we obtain that equation \eqref{eq:06051} is equal to $-\frac{2}{3}(\lambda-1)^{\frac{5}{3}}\lambda^{\frac{4}{3}},$
equation \eqref{eq:06052} is equal to $-\frac{2}{3}\lambda^{\frac{2}{3}}(\lambda-1)^{\frac{1}{3}},$ equation \eqref{eq:06053} is equal to $-\frac{1}{3}(z-1)^{\frac{1}{3}}z^{\frac{5}{3}}(z+1)$ and equation \eqref{eq:06054} is equal to $-\frac{4}{3}(z-1)^{\frac{1}{3}}z^{\frac{8}{3}}.$ 

\end{remark}

\begin{remark} 
The differential forms used in Propositions \ref{prop:algtra0906} and \ref{prop:algtra11407}  are all forms such that $\frac{\omega_\beta}{f^2}$ is a good form, $\frac{\omega_\beta}{f}$ is not a good form, with $A_\beta<2$ and $A_{\beta'}\notin\N.$ We would like to get more algebraic expressions of hypergeometric functions such that the hypergeometric functions are not algebraic. One possible path would be to explore the integrals of good forms $\frac{\omega_\beta}{f^2}$ with $\frac{\omega_\beta}{f}$ is not good form, $A_\beta>2$ and $A_{\beta'}\notin\N.$ The following proposition tells us that such a path is not possible. 
\end{remark}

\begin{prop}
\label{prop:goodform2909}
Consider $f=g(x)+P(y)$, where $g(x)=x_1^{m_1}+\dots+x_{n}^{m_n}$  and  $P(y)=y(1-y)(\lambda-y)$. Suppose that $\frac{\omega_\beta}{f^k}$ is a good form with $A_\beta>k$, then $\frac{\omega_\beta}{f^{k-1}}$ is a good form.
\end{prop}

\begin{proof}
Since $\frac{\omega_\beta}{f^k}$ is a good form, we can write

\begin{equation}
\label{equa:good1306}
\frac{\omega_\beta}{f^k}=\sum C_{k_j}\frac{\omega_{\beta+(0',k_j)}}{f^j},   
\end{equation}
with $C_{k_j}\in\C[\lambda]$ and  $A_{\beta+(0',k_j)}=A_\beta+\frac{k_j}{3}<j$ such that $j-1<A_\beta+\frac{k_j-1}{3}$ or $j-1<A_\beta+\frac{k_j-2}{3}$ (see Remark \ref{rem:incrpole2508}). Now let us apply the process of pole order increment to the differential form $\frac{\omega_\beta}{f^{k-1}}$. We obtain 

\begin{equation*}
\frac{\omega_\beta}{f^{k-1}}=\sum \hat{C}_{k_j}\frac{\omega_{\beta+(0',k_j)}}{f^{j-1}},   
\end{equation*}
where $j-1<A_{\beta+(0',k_j)}<j$. This means that we need to increment the pole order again. 
Let us see what happens 
when $j-1<A_\beta+\frac{k_j-1}{3}$. We have

$$\left[\frac{\omega_{\beta+(0',k_j)}}{f^{j-1}}\right]=\frac{A_\beta+\frac{k_j}{3}}{3(A_\beta+\frac{k_j}{3}-j)}\left[\frac{-a\omega_{\beta+(0',k_j+2)}-2b\omega_{\beta+(0',k_j+1)}}{f^j} \right].$$
We must analyze each term of the previous expression. Let us see the most problematic term: $\frac{\omega_{\beta+(0',k_j+1)}}{f^j}$. Observe that $ j-\frac{1}{3}<A_{\beta+(0',k_j+1)}<j+\frac{1}{3}$. If $A_{\beta+(0',k_j+1)}=j$, then $A_{\beta+(0',k_j-2)}=j-1$. This implies that $\frac{\omega_\beta}{f^k}$ is not good form or the differential form $\frac{\omega_{\beta+(0',k_j-1)}}{f^{j-1}}$ appears  one step before obtaining equation (\ref{equa:good1306}). If we have $\frac{\omega_{\beta+(0',k_j-1)}}{f^{j-1}}$ as $A_{\beta+(0',k_j-1)}=j-\frac{2}{3}$ we need to apply the process of pole order increment again but by applying it we get that $\frac{\omega_\beta}{f^k}$ is not a good form because appears the differential form $\frac{\omega_{\beta+(0',k+1)}}{f^j}$ and $A_{\beta+(0',k_j+1)}=j$. In conclusion $A_{\beta+(0',k_j+1)}\neq j$. If necessary we increment the pole order again. The other cases are similar, leading us to conclude that $\frac{\omega_\beta}{f^{k-1}}$ is a good form. \end{proof}

\begin{remark}
Proposition \ref{prop:goodform2909} tells us that in Proposition \ref{prop:algform1807}, the condition that $\frac{\omega_\beta}{f}$ is a good form is not necessary for $k\leq \frac{n}{2}-1$.
\end{remark}

To obtain more algebraic expressions of hypergeometric functions such that the hypergeometric functions are not algebraic we integrate a strong generic Hodge cycle in a good form $\frac{\omega_\beta}{f^k}$ such that $\frac{\omega_\beta}{f}$ is not good form, $A_\beta<k$ and $A_{\beta'}\notin\N$, where $f=g(x)+P(y)$, $g(x)=x_1^2+x_2^d$ and  $P(y)=y(1-y)(\lambda-y).$ Indeed in the proof of Proposition \ref{prop:algform1807} we saw that up to multiplication by an element of $\overline{\Q(\lambda)}$ 

\begin{equation}
\label{eq:algn10209}
\begin{split}
\frac{1}{2\pi i}\int_{\delta^{0}} res\left(\frac{\omega_\beta}{f^k}\right)
=\frac{B\left(\frac{1}{2},\frac{\beta_2+1}{d} \right)B\left(
A_{\beta'},A_{\beta'}\right)}{2\pi i}\sum_{j}C_j(\lambda)F\left(a_j,b_j,c_j;\frac{1}{\lambda} \right),
\end{split}
\end{equation}
with $F\left(a_j,b_j,c_j;\frac{1}{\lambda} \right)$ contiguous to $F\left(a,b,c;\frac{1}{\lambda} \right)$, where $a=A_{\beta'}+\beta_{3},$       $b=1-A_{\beta'},$         $c=2A_{\beta'}+\beta_{3}$. Therefore if $\delta^0=\sum_{j=0}^{d-2}n_j\delta_0*\delta_{j}^{-1}$ is a generic Hodge cycle we have that \eqref{eq:algn10209} belongs to $\overline{\Q(\lambda)}$. Furthermore, if  

\begin{equation*}
\sum_{j=0}^{d-2}n_{j}\zeta_{d}^{j(\beta_2+1)}\neq 0,
\end{equation*}
using equation \eqref{eq:alg2301} we conclude that $\sum_{j}C_j(\lambda)F\left(a_j,b_j,c_j;\frac{1}{\lambda} \right)\in\overline{\Q(\lambda)}.$ The fact that $\frac{\omega_{\beta}}{f}$ is not a good form implies that $A_{\beta'}=\frac{N}{3}$ for some $N\in\N$, and therefore $F\left(a,b,c;\frac{1}{\lambda} \right)$ is reducible. Also note that

$$a_j=a+k_j,\;\; b_j=b+l_j,\;\; c_j=c+d_j,$$
with $k_j,l_j,d_j\in\Z.$ Consider $\lambda_j=1-c_j$, $\mu_j=c_j-a_j-b_j$, $\nu_j=a_j-b_j$. A straightforward computation allows us to verify that $\lambda_j,\mu_j,\nu_j$ do not satisfy the hypothesis of \cite[Theorem 12.17, item (c)]{zolkadek2006monodromy} and therefore $F\left(a_j,b_j,c_j;\frac{1}{\lambda} \right)\notin\overline{\Q(\lambda)}.$ The same is valid for the cycle $\delta^{1}.$

\subsection{Computational verification}
\label{sec:compu2408}

We can check the validity of Propositions \ref{prop:algtra0906} and \ref{prop:algtra11407} using numerical computations by evaluating $\lambda$ at algebraic numbers. Call $G(\lambda)$ the function defined by equation \eqref{eq:algHF1808}. We use the package\\
\texttt{with(IntegerRelations)} in \textbf{Maple}. The command

\begin{verbatim}
    v := expand([seq(evalf[k](G(lambda)^j), j = 0 .. m)]);
\end{verbatim}
computes powers of $G(\lambda)$ from $0$ to $m$ with $k$ digits of precision. With the following command

\begin{verbatim}
    u := LinearDependency(v, method = LLL);
\end{verbatim}
we find a $\Z$-linear relation between $1,G(\lambda),G(\lambda)^2,\dots,G(\lambda)^m$. The polynomial that satisfies $G(\lambda)$ can be displayed with the command 

\begin{verbatim}
    P := add(u[j]*z^(j-1), j = 1 .. m+1);
\end{verbatim}
This computation is heuristic, since we only have approximations of $G(\lambda)$. As an example of the above take $\lambda=i$ with $i^2=-1$, $m=400$ and $400$ digits of precision. We have the polynomial 
\begin{equation*}
     81z^4-900z^2+10000.
\end{equation*}
These computations suggest that $G(i)$ is an algebraic number. This is, of course, one consequence of Proposition \ref{prop:algtra0906}. In fact, by Remark \ref{rem:1610} we know that $G(\lambda)=\frac{10}{3}\lambda^{\frac{5}{3}}$. Observe that $G(i)$ is root of $81z^4-900z^2+10000.$

With these same commands we can verify what was proven by Reiter and Movasati in \cite{movasati2006hypergeometric} mentioned in the introduction of this paper.